\documentclass{amsproc}

\numberwithin{equation}{section}
\usepackage{enumerate}

\newcommand{\N}{\mathbb{N}}
\newcommand{\E}{\mathbb{E}}

\renewcommand{\H}{\mathbb H}
\renewcommand{\P}{\mathbb{P}}

\newcommand{\bi}{\mathbf{i}}
\newcommand{\bs}{\mathbf s}

\newcommand{\LL}{\mathcal L}

\renewcommand{\P}{\mathbb P}

\usepackage{ulem}

\newcommand{\FFF}{\mathcal F}

\newcommand{\<}{\langle}
\renewcommand{\>}{\rangle}

\newcounter{RomanNumber}

\newtheorem{tm}{Theorem}[section]
\newtheorem{df}{Definition}[section]
\newtheorem{lm}{Lemma}[section]
\newtheorem{prop}{Proposition}[section]
\newtheorem{cor}{Corollary}[section]
\newtheorem{rk}{Remark}[section]

\newtheorem{ap}{Assumption}[section]

\allowdisplaybreaks[4]

\begin{document}

\title[Explicit numerical methods for SNLSEs]
{Explicit Numerical Methods for High Dimensional Stochastic Nonlinear Schr\"odinger Equation: Divergence, Regularity and Convergence}

\author{Jianbo Cui}
\address{Department of Applied Mathematics, The Hong Kong Polytechnic University, Hung Hom, Hong Kong}
\curraddr{}
\email{jianbo.cui@polyu.edu.hk}
\thanks{The research of is partially supported by start-up funds (P0039016) from Hong Kong Polytechnic University and the CAS AMSS-PolyU Joint Laboratory of Applied Mathematics.}



\subjclass[2010]{Primary 60H35; Secondary  35Q55, 60H15, 65M12, 65C30, 65P10}

\keywords{stochastic nonlinear Schr\"odinger equation, explicit numerical method, divergence, regularity estimate, tail estimate, strong convergence.}

\date{\today}

\dedicatory{}

\begin{abstract}
This paper focuses on the construction and analysis of explicit numerical methods of high dimensional stochastic nonlinear Schr\"odinger equations (SNLSEs). 
We first prove that the classical explicit numerical methods are unstable and suffer from the numerical  divergence phenomenon.  Then we propose a kind of explicit splitting numerical methods and prove that the structure-preserving splitting strategy is able to enhance the numerical stability.  Furthermore, we establish the regularity analysis and strong convergence analysis of the proposed schemes for SNLSEs based on two key ingredients.  One ingredient is proving new regularity estimates
of SNLSEs by constructing a logarithmic auxiliary functional and exploiting the Bourgain space.  Another one is providing a dedicated error decomposition formula and a novel truncated stochastic Gronwall's lemma, which relies on the tail estimates of underlying stochastic processes.
In particular, our result answers the strong convergence problem of numerical methods for 2D SNLSEs emerged from [C. Chen, J. Hong and A. Prohl, Stoch. Partial Differ. Equ. Anal. Comput. 4 (2016), no. 2, 274–318] and  [J. Cui and J. Hong,  SIAM J. Numer. Anal. 56 (2018), no. 4, 2045–2069].
\end{abstract}

\maketitle


\section{Introduction}
\label{int}
Consider the following SNLSEs which model the propagation of nonlinear dispersive waves in an inhomogeneous or random media (see e.g. \cite{MR1425880,BCI95,FKLT01} and references therein): 
\begin{align}\label{SNLS}
d u&=\bi \Delta u dt+ f(u) dt-\frac 12 \alpha u dt+g(u) dW(t),\\\nonumber 
u(0)&=\varPsi,
\end{align}
where $ f(\xi)=\bi \lambda |\xi|^{2\sigma}\xi$ with $\sigma>0$ and the parameter $\lambda=1$ and $-1$ corresponding to the focusing case and defocusing case in physics, respectively. Here the spatial domain $\mathcal O$ is either a bounded Lipschitz domain in $\mathbb R^d, d\le 2$ equipped with a suitable boundary condition (such as the homogeneous Dirichlet or Neumann boundary condition) or a compact Riemannian manifold of $d\ge 2$ without boundary. 
The operator $g$ is either the Nemytskii operator of the constant function, which is related to the additive noise case, or the Nemytskii operator of $\bi \xi$ which corresponds to the multiplicative noise case. The diffusion term $g(u) dW(t)$ represents the fluctuation effect of a physical process in the complex media \cite{BD03}, where $\{W(t)\}_{t\ge 0}$ is an $L^2(\mathcal O;\mathbb C)$-valued $Q$-Wiener process on a stochastic basis $(\Omega,\mathcal F, \{\mathcal F_t\}_{t\ge 0},\mathbb P)$. This implies that there exists an orthonormal basis $\{e_j\}_{j\in \mathbb N}$ of   $L^2(\mathcal O;\mathbb C)$ and a sequence of mutually independent Brownian motions $\{\beta_j\}_{j\in \mathbb N}$ such that $W(t,x)=\sum\limits_{j} Q^{\frac 12}e_j(x)\beta_j(t).$
The real-valued function $\alpha(x), x\in\mathbb R$, which measures the damping effect during the propagation of waves over long distance, is set to be smooth enough and may depend on the covariance operator of $W$ \cite{CH17}. 
 
SNLSEs have been investigated both theoretically and numerically in the recent decades.
For the well-posedness in $L^2(\mathcal O;\mathbb C)$ and $H^1(\mathcal O;\mathbb C)$ and  the global and asymptotic behaviors of SNLSEs, we refer to \cite{BD03,MR3232027,MR3980316,CHS18b,CS21} and references therein.  Among them, the numerical approximation turns out to be a useful and important tool (see, e.g., \cite{DD02a, MR4333509}) since it is impossible to find the analytic solution of  Eq. \eqref{SNLS} in general. 
For SNLSEs with Lipschitz and smooth coefficients, there exist fruitful numerical results  on the stability and strong convergence (see, e.g., \cite{BD06,AC18}), and on the structure-preserving properties and longtime dynamic behaviors (see, e.g., \cite{JWH13,CH16}). 
However, the basic mathematical mechanism, such as the stability and strong convergence, of numerical methods for SNLSEs with non-monotone coefficients, like \eqref{SNLS}, has not been completely understood. To deal with the strong nonlinearity in SNLSEs, 
many authors use the stopping time techniques and truncated SNLSEs to consider the convergence rates of numerical methods in probability or in pathwise sense (see, e.g., \cite{BD06,Liu13,CHP16}) which is weaker than the strong one.  Some progress has been achieved by studying exponential integrability of  exact and numerical solutions (see, e.g., \cite{CHL16b,CH17,CHLZ19,BC20}).  For 1D stochastic cubic Schr\"odinger equation, the authors in \cite{CHL16b,CH17,CHLZ19} derive the optimal strong and weak convergence rates of a kind of temporal splitting Crank-Nicolson schemes and their full discretizations. 
Nevertheless, the convergence problem of numerical methods for SNLSEs with general polynomial nonlinearity in higher dimensions remains open (see, e.g., \cite{CHP16,Liu13,CH17}), which is one main motivation of this work. 
  
Another motivation lies on the divergence phenomenon and instability of explicit numerical methods for stochastic differential equations (SDEs) with superlinear coefficients \cite{MR2795791}. It has been also shown that for stochastic partial differential equations (SPDEs) of parabolic type with polynomial nonlinearities, the moments of exponential and linear-implicit Euler method are divergent \cite{BHJKLS19}.       
Naturally, we are inspired to ask whether the divergence phenomenon of explicit numerical schemes also exists for SPDEs of hyperbolic type.  This finding may be used to explain why the structure-preserving strategy, the truncated or tamed technique, as well as the adaptive method, are needed and important in designing numerical schemes for SPDEs of hyperbolic type with superlinear coefficients. 

In this paper, we show the divergence of several explicit temporal numerical schemes, which include the classical exponential Euler method, for Eq. \eqref{SNLS}, based on the property that the double-exponent moment of the Wiener process is infinite  (see Section 3). To overcome the divergence issue, we use the structure-preserving idea to construct a kind of explicit numerical methods by some suitable spatial discretizations. 
Via the Lie--Trotter splitting technique, we propose the following explicit structure-preserving splitting scheme, which is unconditionally stable in $L^{p}(\Omega;\mathbb H), p\in \mathbb R$ (see Proposition \ref{Sta1}),
\begin{align*}
u_{n+1}^M&=\Phi_{S,n}^M(\delta t, \Phi_{D,n}^{M}(\delta t, u_n^M)),
\end{align*}
where $\Phi_{D,n}^{M},\Phi_{S,n}^{M}, t\in [t_n,t_{n+1}]$ are the phase flows of the following subsystems 
\begin{align*}
&dv_{D,n}^M(t)=\bi \Delta  v_{D,n}^M (t)dt, \; v_{D,n}^M(t_n)= u^M_{n}, u_{S,n}(t_n)=v_{D,n}^M(t_n), \\
&du_{S,n}^M(t)=P^M\Big(\bi \lambda |u_{S,n}^M(t_n)|^{2\sigma}u_{S,n}^M(t)-\frac 12\alpha u_{S,n}^M(t)\Big)dt+ P^M g(u_{S,n}^M(t)) dW(t).
\end{align*}  
Here $u_0=\varPsi$, $\delta t$ is the time stepsize, $\delta_n W:=W(t_{n+1})-W(t_n), n\le N-1, N \delta t=T, $ $M$ is the parameter of the spectral Galerkin projection operator $P^M$.   
One can also use the truncated strategy to construct stable explicit numerical schemes, such as the nonlinearity-truncated exponential Euler method,
{\small 
\begin{align*}
u_{n+1}^M=S(\delta t) P^M u_n^M+\mathbb I_{\{\|u_{n}^M\|_{\mathbb H^{\kappa}}\le R \}}S(\delta t) P^M (f(u_n^M)-\frac 12 \alpha u_n^M)\delta t +S(\delta t) P^M g(u_{n}^M)\delta_n W,
\end{align*}
}where $S(t)=\exp(\bi \Delta t)$, $R$ is the truncated number and $\kappa \ge 0$ is the Sobolev index in the truncated function. We would like to remark that the spatial discretization can be also chosen as  the finite element methods, the finite difference method, as well as other methods with suitable inverse inequalities.

In the second part of this work, we are interested in the strong convergence of the proposed numerical methods for Eq. \eqref{SNLS}. At the outset of the convergence problem, there is no a higher Sobolev regularity estimate, i.e., $\mathbb H^{\bs}$-estimate with $\bs>1$, of both exact  and numerical solutions  for Eq. \eqref{SNLS} in $d\ge 2$.  (see, e.g.,  \cite[Remarks 2 and 3]{CHP16}).  Up to now, the existing results on higher regularity estimates for the exact or numerical solution of SNLSEs are only limited to 1D case. To solve the regularity problems, we introduce the logarithmic Sobolev inequality to construct a new Lyapunov functional and establish the higher regularity estimate for 2D SNLSEs. For SNLSEs on a compact Riemannian manifold of $d\ge 2$, by imposing additional assumptions on $W$, we exploit the Bourgain space and its nonlinear estimate (see, e.g., \cite{BGT05}) to show the desirable higher regularity almost surely (see Section 2). With regard to the numerical solution, we prove that there always exists a relationship between $\delta t$ and the eigenvalue  $\lambda_M$ of the Laplacian operator such that the splitting numerical method could inherit the regularity of SNLSEs (see Section 4).  

Another bottleneck problem in strong convergence analysis for SPDEs with non-monotone coefficients, including Eq. \eqref{SNLS}, is the lack of a systematic way to transform the pointwise nonlinear error estimate into the desirable strong error estimate. Several attempts have been made on this topic in the recent years. For instance, the approach based on the exponential integrability and stochastic Gronwall's inequalities of the exact and numerical solutions shows the ability to deal with a large class of non-monotone stochastic ordinary differential equations \cite{HJ14}. In contrast, several restrictions on the dimension and nonlinearity have to be imposed for non-monotone SPDEs via this approach.  To tackle these issues, we provide a new error decomposition formula which relies on the tail estimates of the underlying stochastic processes (see Section 5).  These tail estimates are naturally obtained via the established higher regularity of numerical and exact solutions. As a consequence, we prove the strong convergence of the proposed numerical methods for SNLSEs, including 1D SNLSs (Theorem \ref{tm-1d}), 2D stochastic cubic  Schr\"odinger equation (Theorem \ref{tm-2d}), and higher dimensional cubic Schr\"odinger equations with random coefficients (Theorem \ref{3d-tm}).
 In particular, our result gives a positive answer to the open problem on the strong convergence of 2D SNLSEs emerged from \cite{CHP16,Liu13,CH17}.  
The proposed approach also has a potential to be extended to deal with the convergence of numerical schemes for other non-monotone SPDEs, such as the stochastic nonlinear wave equation, stochastic Korteweg--De Vries equation and stochastic Burgers equation, and this will be investigated in the future.

To conclude, the main contributions of this paper are summarized as follows:
\begin{itemize}
\item  We show the divergence of explicit numerical methods for SNLSEs, and prove that 
the splitting strategy is able to enhance the stability and regularity of the numerical solution.

\item  We provide a generic approach to study the strong convergence of  numerical methods for SDEs with non-monotone coefficients based on the tail estimates and a truncated stochastic Gronwall's lemma.

\item  We establish the convergence analysis of the proposed explicit splitting numerical methods via proving new $\mathbb H^2$-regularity of stochastic nonlinear Schr\"odinger equation in high dimensions.
\end{itemize}


\section{SNLSEs: Higher regularity estimates and tail estimates}
\label{high-reg-tail}

In this section, we present the basic notations and definitions, and some new properties for SNLSEs, including the higher Sobolev regularity estimates and tail estimates of 1D SNLSEs with general polynomial nonlinearities, $\mathbb H^{2}$-regularity estimates and tail estimates for 2D stochastic cubic  Schr\"odinger equation and  defocusing cubic Schr\"odinger equations with random coefficients on a compact Riemannian manifold. Throughout this paper, we use $C$ to denote a generic constant which is independent of $\delta t, M, R$ and may differ from line to line.

\subsection{Preliminaries} 
In this part, we introduce some frequently used notations and assumptions in the study of SNLSEs. The norm of $\mathbb H:=L^2(\mathcal O;\mathbb C)$ is denoted by  $\|\cdot\|$ and  the inner product $\<\cdot, \cdot\>$ is defined by $\<v,w\>:=Re \int_{\mathcal O} \bar v(x) w(x) dx$ for $v,w\in  \mathbb H.$ We denote $L^p:=L^p(\mathcal O;\mathbb C)$ and $W^{k,p}:=W^{k,p}(\mathcal O;\mathbb C), k\in \mathbb R, p\ge 1.$ When $p=2,$ $W^{k,2}$ is denoted by $H^{k}.$ We denote  the interpolation Sobolev space of the corresponding Laplacian operator on $\mathcal O$ by $\mathbb H^{\gamma}, \gamma \in \mathbb R$ and the corresponding eigenvalue by $\{\lambda_i\}_{i\in \mathbb N}.$ 
An operator $\Psi \in \mathcal L_2^{\bs}$ with $\bs\in \mathbb N$ if  $\|Q^{\frac 12}\|_{\mathcal L_2^{\bs}}:=\sum_{i}\|Q^{\frac 12}e_i\|^2_{\mathbb H^{\bs}}<\infty,$
where $\{e_i\}_{i\in \mathbb N}$ is any orthonormal basis of $\mathbb H$. For simplicity, we assume that $\sup\limits_{i}\|e_i\|_{L^{\infty}}<\infty.$ 
The length of the time interval is a positive number $T>0$. 

Let us briefly recall the previous well-posedness result of SNLSEs before we state our main assumptions. 
The mild solution of Eq. \eqref{SNLS} is defined by the following integral equation
\begin{align*}
u(t)&=S(t)u_0+\int_{0}^tS(t-s)f(u(s))ds-\int_{0}^tS(t-s)\frac 12 \alpha u(s)ds\\
&\quad+\int_{0}^tS(t-s)g (u(s))dW(s), \; a.s.,
\end{align*}
where $S(t):=\exp(\bi \Delta t)$ is the group generated by $\bi \Delta$. 
For the global well-posedness of the mild solution of Eq. \eqref{SNLS}, we refer to, e.g., \cite{BD03}, in the subcritical case, i.e., $\sigma<\frac {2}{(d-2)^+}$ if $\lambda=-1$ and $\sigma<\frac {2}d$ if $\lambda=1$. It is also known that in the additive noise case $(g(\xi)=1)$, the mass and energy conservation laws fail and that in  multiplicative noise $(g(\xi)=\bi \xi),$ the mass conservation law holds if $W(\cdot)$ is $L^2(\mathcal O;\mathbb R)$-valued and $\alpha=\sum\limits_{i} |Q^{\frac 12}e_i|^2$  (see e.g. \cite{BD03}).  We are also interested in the multiplicative focusing critical case ($g(\xi)=\bi \xi, \lambda=1, \sigma=\frac 2d$)  where the initial value is required to satisfy $\|\varPsi\|<C_{Thr},$ where $C_{Thr}$ is the $L^2$-norm of the ground state solution of the elliptic equation $\Delta v-v+v^{2\sigma+1}=0$ (see e.g. \cite{MR691044,CHS18b}). 
Among these works, the evolution of the energy $$H(w):=\frac 12 \|\nabla w\|^2+\frac 1{2\sigma+2}\|w\|_{L^{2\sigma+2}}^{2\sigma+2}$$ plays a key role.

We will assume that the initial value $\varPsi$ is deterministic unless it is necessary to avoid confusion.
To simplify the presentation, we impose the following conditions on the diffusion term and nonlinearity which will be frequently used in this section and section 4.

\begin{ap}\label{add}
Let $g(\xi)=1$, $\alpha=0$, $\sigma\in \mathbb N^+$, $\varPsi\in \mathbb H^{\bs}$ and $Q^{\frac 12}\in \mathcal L_2^{\bs}$ for some $\bs\in \mathbb N$. 
Suppose that one of the following condition holds,
\begin{enumerate}[(i)]
\item when $d=1$, it holds that $\sigma=1$ if $\lambda=1$ and $\sigma\in \mathbb N^+$ if $\lambda=-1.$
\item when $d=2$, it holds that $\sigma=1$ and $\lambda=-1$.
\end{enumerate}
\end{ap}

\begin{ap}\label{mul}
Let $g(\xi)=\bi \xi,$ $\alpha=\sum_{i} |Q^{\frac 12}e_i|^2$, $\sigma\in \mathbb N^+,$ $\varPsi\in \mathbb H^{\bs}$, $W(\cdot)$ be $L^2(\mathcal O;\mathbb R)$-valued and $Q^{\frac 12}\in \mathcal L_2^{\bs}$ satisfy $\sum_{i} \|Q^{\frac 12}e_i\|^2_{W^{\bs,\infty}}<\infty$ for some $\bs\in \mathbb N$.
Suppose that one of the following condition holds, 
\begin{enumerate}[(i)] 
\item when $d=1$, it holds that $\sigma\le 2,$
\item when $d=2$, it holds that $\sigma=1.$
\end{enumerate} 
Furthermore, if $\sigma=\frac 2 d, \lambda=1,$ we in addition assume that 
$\|\varPsi\|< C_{Thr}.$
\end{ap}

With a slight modification, one could follow our approach and investigate the additive noise case with linear damping effect $\alpha\neq 0$ and the multiplicative noise case with the complex-valued Wiener process. Our assumptions $\sigma\in \mathbb N^+$ could be also extended to general positive real number, i.e., $\sigma>0$ if $d=1$ and $\sigma>\frac 12$ if $d=2$, but there still exist some limitations on the upper bound of $\sigma$ since the global $\mathbb H^2$-regularity estimate of SNLSEs is still unclear when $\sigma$ is large.  
In the following, we collect the higher regularity results for 1D SNLSEs and  present some new results for higher dimensional SNLSEs.

\subsection{1D stochastic nonlinear Schr\"odinger equation}

In this part, we show the regularity results of 1D SNLSEs with general polynomial nonlinearities.  This higher regularity result could be obtained by studying the auxiliary functional $V(v)=\|(-\Delta )^{\frac \bs 2}v\|^2-\lambda \<(-\Delta)^{\bs-1} v, |v|^{2\sigma}v\>$ (see, e.g., \cite[Section 2]{CHL16b}), and thus we omitted its proof. Part of them has been reported in \cite{CHL16b}.

\begin{prop}\label{1d-prop}
Let $T>0$, $d=1$, $\bs \in \mathbb N^+$, $\varPsi\in \mathbb H^{\bs}, \sigma \in \mathbb N^+$, $Q^{\frac 12}\in \mathcal L_2^\bs$. 
Suppose that $\sigma\in (0,2),\lambda=1$ or $\sigma\in (0,\infty),\lambda=-1$ if $g(\xi)=1$ and $\alpha=0$, and that $\sigma\in (0,2], \lambda=1$ or  $\sigma\in (0,\infty),\lambda=-1$ if $g(\xi)=\bi \xi, \alpha=\sum\limits_{i} |Q^{\frac 12}e_i|^2$ and $W(\cdot)$ is $L^2(\mathcal O;\mathbb R)$-valued.
Furthermore, if $g(\xi)=\bi \xi$, $\lambda=1, \sigma=2,$ in addition assume that $\|\varPsi\|<C_{Thr}.$
There exists a unique mild solution of \eqref{SNLS} satisfying 
\begin{align*}
\E \Big[ \sup_{s\in [0,T]} \|u(s)\|_{\mathbb H^{\bs}}^{p}\Big]\le C(T,Q,\varPsi,\lambda,\sigma,p).
\end{align*}  
\end{prop}

Similar to the deterministic case \cite{MR2002047}, it is not hard to check that to obtain $\mathbb H^2$-regularity, one needs to take $\sigma\ge \frac 12, \sigma\in \mathbb R^+$. 
Proposition \ref{1d-prop} implies the following useful tail estimates of the mild solution under different norms. 

\begin{cor}\label{1d-cor}
Under the condition of Proposition \ref{1d-prop}, it holds that for a large $R>0,$
\begin{align}\label{exp-l2}
\mathbb P\big(\sup_{s\in [0,T]} \|u(s)\| \ge R\big)
&\le 
C(T,Q,\varPsi,\lambda,\sigma)\exp(-\eta R^2)
\end{align}
with some $\eta=\eta(T,Q,\varPsi,\lambda,\sigma)>0$, 
and 
\begin{align}\label{poly-hs}
\mathbb P\big(\sup_{s\in [0,T]} \|u(s)\|_{\mathbb H^{\bs}} \ge R_1\big)
&\le 
C(T,Q,\varPsi,\lambda,\sigma,p_1)R_1^{-p_1}, \; \forall \; p_1\in \mathbb N^+.
\end{align}
Furthermore, if $g(\xi)=\bi \xi$, then for large $R_2,R_3>0,$
there exist $\eta_1=\eta_1(T,Q,\varPsi,\lambda,\sigma)$, $\eta_2=\eta_2(T,Q,\varPsi,\lambda,\sigma)>0$ such that 
\begin{align}\label{exp-h1}
\mathbb P\big(\sup_{s\in [0,T]}  \|u(s)\|_{\mathbb H^{1}}\ge R_2 \big)
&\le C(T,Q,\varPsi,\lambda,\sigma)\exp(-\eta_1 R_2^2),
\end{align}
and that 
\begin{align}\label{exp-lin}
\mathbb P\big(\sup_{s\in [0,T]}  \|u(s)\|_{L^{\infty}} \ge R_3 \big)
&\le C(T,Q,\varPsi,\lambda,\sigma)\exp(-\eta_2 R_3^4).
\end{align}
\end{cor}

\begin{proof}
The tail estimate  \eqref{exp-l2} in the additive noise case could be obtain by applying \cite[Lemma 3.1]{CHL16b} to the mass functional $\|u(t)\|^2.$ In the multiplicative noise case, the mass conservation law immediately implies \eqref{exp-l2}. The Chebyshev inequality, together with Proposition \ref{1d-prop}, leads to \eqref{poly-hs}.
Applying \cite[Lemma 3.1]{CHL16b} to the energy functional $H(u(t))$ and using the the mass conservation law, one can obtain the exponential integrability $\E\Big[\sup\limits_{t\in [0,T]}\exp(e^{-\alpha_1 t} \|u(t)\|^2)\Big] <\infty$ for some $\alpha_1=\alpha_1(\Psi,Q,\lambda,\sigma,T).$   
Using the arguments in the proof of \cite[Corollary 4.1]{CHLZ19},  we obtain \eqref{exp-h1} where  $\eta_1=e^{-\alpha_1 T}$.
By using the Gagliardo--Nirenberg interpolation inequality,  there exists $C'>0$ such that $\|u\|_{L^{\infty}}\le C' \|\nabla u\|^{\frac 12} \|u\|^{\frac 12}.$ It follows that 
\begin{align*}
\mathbb P\Big(\sup_{s\in [0,T]}  \|u(s)\|_{L^{\infty}} \ge R_3 \Big)
&\le \mathbb P\Big(\sup_{s\in [0,T]} \|\nabla u(s)\| \ge \frac {R_3^2}{(C')^2\|u(0)\|} \Big)\\
&\le C(T,Q,u_0,\lambda) \exp\Big(-\frac {\eta_1}{(C')^4\|u(0)\|^2} R_3^{4}\Big),
\end{align*}
which completes the proof of \eqref{exp-lin}.
\end{proof}

\subsection{2D stochastic cubic Schr\"odinger equation}
In this part, we present the higher regularity estimate and tail estimate for 2D stochastic cubic Schr\"odinger equation, which has not been reported in the literature. 
To study the higher regularity,
our key tool is the following critical Sobolev interpolation inequality whose proof is in the appendix.
\begin{lm}\label{cri-sob}
Let $v\in \mathbb H^2$. It holds that for some $C_0>0,$
\begin{align*}
\|v\|_{L^{\infty}}&\le C_0\|v\|_{\mathbb H^1}\big(1+\sqrt{\log(1+\|v\|_{\mathbb H^2}^2)}\big).
\end{align*}
\end{lm}

For simplicity, let us assume that $\mathcal O$ is  a bounded domain equipped with homogeneous Dirichlet boundary condition. We introduce a new auxiliary functional $\widetilde U$ defined by 
$$\widetilde U(w)=\log(1+\log(1+\|\Delta w\|^2)).$$ 
We would like to remark when considering the Cauchy problem posed on a bounded domain $\mathcal O$ equipped with  homogeneous Neumann boundary condition or periodic boundary condition, one need to consider $\widetilde U(w)=\log(1+\log(1+\|w\|_{\mathbb H^{2}}^2))$ for studying the higher regularity. 

\begin{prop}\label{d=2-h2}
Let $d=2$ and $T>0$. Suppose that Assumption \ref{add} or \ref{mul} holds with $\bs=2$.
The mild solution $u$ satisfies that for $p\in \mathbb N^+,$
\begin{align*}
\E \Big[\sup_{t\in[0,T]}\widetilde U^p(u(t))\Big]\le C(T,Q,\varPsi,\lambda,p).
\end{align*}
\end{prop}

\begin{proof}
We only present the proof of the multiplicative noise case since the proof of the additive noise case is similar and simpler.
According to our assumptions, by using the energy evolution (see, e.g., \cite{BD03,CHL16b}), one can obtain the following boundedness of any finite $p$-th moment, 
\begin{align}\label{pri-ene}
\E \Big[\sup_{t\in [0,T]}\|u(t)\|^{2p}_{\mathbb H^1}\Big]\le C(T,Q,\varPsi,\lambda,p).
\end{align} 
Applying It\^o's formula to $\widetilde U^p(u(t)),$ we obtain that 
 \begin{align*}
 &\widetilde U^p(u(t))\\
 &=\widetilde U^{p}(u(0))
+ \int_0^t p 2\widetilde U^{p-1}(u)  \frac 1{1+\log(1+\|\Delta u\|^2)}\frac 1{1+\|\Delta u\|^2}\Big(\<\Delta u, \bi \lambda 2Re(\bar u \Delta u)u \\
&\quad+\bi \lambda 4Re(\bar u \nabla u)\nabla u+\bi \lambda 2 |\nabla u|^2 u\>\Big)ds\\
&-\int_0^t p \widetilde U^{p-1}(u) \frac 1{1+\log(1+\|\Delta u\|^2)} \frac 1{1+\|\Delta u\|^2}\sum_{i\in\mathbb N^+}\Big( \<\Delta u,2 |\nabla  Q^{\frac 12}e_i|^2  u\\
&\quad +2\nabla u \nabla Q^{\frac 12} e_i e_i +\Delta u |Q^{\frac 12}e_i|^2+2u\Delta Q^{\frac 12}e_iQ^{\frac 12}e_i\> \Big)ds\\
&+\int_0^t 2p\widetilde  U^{p-1}(u) \frac 1{1+\log(1+\|\Delta u\|^2)}  \frac 1{1+\|\Delta u\|^2}\< \Delta u,\bi \Delta (udW(s))\>\\
&+\int_0^t p \widetilde U^{p-1}(u) \frac 1{1+\log(1+\|\Delta u\|^2)} \frac 1{1+\|\Delta u\|^2} \sum_{i\in \mathbb N^+}\Big( \|\nabla u \nabla Q^{\frac 12}e_i\|^2+\|\Delta u  Q^{\frac 12}e_i\|^2\\
&\quad +\|u \Delta Q^{\frac 12}e_i\|^2+
2\< u \Delta Q^{\frac 12}e_i, \nabla u \nabla Q^{\frac 12}e_i\>
+2\<\nabla u\nabla Q^{\frac 12}e_i,\Delta u Q^{\frac 12}e_i\>\\
&\quad+2\<u \Delta Q^{\frac 12}e_i,\Delta u Q^{\frac 12}e_i\>\Big)ds\\
&+\int_0^t -2p \widetilde U^{p-1}(u)\frac 1{1+\log(1+\|\Delta u\|^2)} \frac 1{(1+\|\Delta u\|^2)^2} \Big(1+\frac 1{1+\log(1+\|\Delta u\|^2)}\\
&\quad -(p-1)\Big) 
\sum_{i\in \mathbb N^+}\Big(\<\Delta u, \bi \nabla u\nabla Q^{\frac 12}e_i\>+\<\Delta u,\bi u \Delta Q^{\frac 12}e_i\>\Big)^2ds.
 \end{align*}
Taking supreme over $t\in[0,t_1]$ and taking expectation, applying H\"older's and Young's inequalities, using the Gagliardo--Nirenberg interpolation inequality  and  Lemma \ref{cri-sob}, as well as Burkerholder's inequality, we obtain that for a small $\epsilon\in (0,1),$ 
 \begin{align*}
 &\E \Big[\sup_{t\in [0,t_1]}\widetilde U^p(u(t)\Big]\\
 &\le \E [ \widetilde U^{p}(u(0))]+C\int_0^{t_1} \E \Big[\widetilde U^{p-1}(u)\frac 1{1+\log(1+\|\Delta u\|^2)} \Big(1+\|u\|_{L^{\infty}}^2 \\
 &\quad +\|u\|_{L^{\infty}}\|\nabla u\|+\|u\|^2 +\|\nabla u\|^2\Big)\Big]ds\\
 & +
\E \Big[\sup_{t\in [0,t_1]}\Big|\int_0^t 2p \widetilde U^{p-1}(u) \frac 1{1+\log(1+\|\Delta u\|^2)}  \frac 1{1+\|\Delta u\|^2}\< \Delta u,\bi \Delta (udW(s))\>\Big|\Big]\\
 &\le  \E [\widetilde U^{p}(u(0))]+C\int_0^{t_1} \E \Big[\widetilde U^{p-1}(u)\Big(1+\|u\|^2 +\|\nabla u\|^2\Big)\Big]ds\\
 &+C\E \Big[\Big(\int_0^{t_1} \widetilde U^{2p-2}(u) (1+\|u\|^2+\|\nabla u\|^2) ds\Big)^{\frac 12}\Big]\\
 &\le \E [\widetilde U^{p}(u(0))]+\epsilon \E [\sup_{t\in [0,t_1]}\widetilde U^p(u(t)]+C(\epsilon)\int_0^{t_1} \E [\widetilde U^{p}(u)]ds\\
 &+C(\epsilon) \int_0^t \E \Big[ 1+\|u\|^{2p} +\|\nabla u\|^{2p} \Big]ds.
 \end{align*}
The Gronwall's inequality and \eqref{pri-ene} yield that 
 \begin{align*}
 \E \Big[\sup_{t\in [0,T]} \widetilde U^p(u(t)\Big]\le C(T,Q,\varPsi,\lambda,p),
 \end{align*}
 which completes the proof.
\hfill
\end{proof}

Thanks to the above regularity estimate, we are able to present the tail estimate for \eqref{SNLS} with $d=2$.

\begin{cor}\label{2d-cor}
Under the condition of Proposition \ref{d=2-h2}, for large $R,R_1>0,$ \eqref{exp-l2} and \eqref{poly-hs} hold with $\bs=1$. 
Furthermore, for  a large $R_2>0,$ it holds that for any $p_1\ge 1,$
\begin{align}\label{2d-exp-lin}
\mathbb P\big(\sup_{s\in [0,T]}  \|u(s)\|_{L^{\infty}} \ge R_2 \big)
&\le  C(T,Q,\varPsi,\lambda,p_1) [\log (R_2)] ^{-p_1},\\\nonumber
\mathbb P\big(\sup_{s\in [0,T]}  \|u(s)\|_{\mathbb H^2} \ge R_2 \big)
&\le  C(T,Q,\varPsi,\lambda,p_1) [\log(1+\log(1+R_2^2))] ^{-p_1}.
\end{align}
Furthermore, if $g(\xi)=\bi \xi$, then \eqref{exp-h1} holds.
\end{cor}
\begin{proof}
The proofs of \eqref{exp-l2}-\eqref{exp-h1} are similar to those in Proposition \ref{1d-prop}. We only prove \eqref{2d-exp-lin}. Applying Lemma \ref{cri-sob}, \eqref{poly-hs}, the Chebyshev inequality and Proposition \ref{d=2-h2}, we obtain that for any $p_1\ge 1,$
\begin{align*}
&\mathbb P\big(\sup_{s\in [0,T]}  \|u(s)\|_{L^{\infty}} \ge R_2 \big)\\
&\le  \mathbb P\Big(\sup_{s\in [0,T]}  \Big[\log (1+\|u(s)\|_{\mathbb H^1}) +\log(1+\log^{\frac 12}(1+\|\Delta u(s)\|^2)) \Big]\ge  \log (R_2)-\log(C) \Big)\\
&\le \mathbb P\Big(\sup_{s\in [0,T]}  \log (1+\|u(s)\|_{\mathbb H^1}) \ge \frac {  \log (R_2)-\log(C)}2\Big) \\
&+ \mathbb P\Big(\sup_{s\in [0,T]} \widetilde U(u(s))  \ge  (\frac {  \log (R_2)-\log(C) }2)^2 \Big)
\\
&\le C(T,Q,\varPsi,\lambda,p_1) [\log (R_2)]^{-p_1}.
\end{align*}
The estimate of $\mathbb P\big(\sup\limits_{s\in [0,T]}  \|u(s)\|_{\mathbb H^2} \ge R_2 \big)$ is similar.
\end{proof}

\subsection{Higher dimensional stochastic cubic Schrodinger equation}

In this part, we suppose that $\mathcal O$ is any compact Riemannian manifold of $d\ge 2$ without boundary. For simplicity, we will only consider the cubic Schrodinger equation with random coefficient, namely,
\begin{align}\label{SCSE}
du=\bi \Delta u dt+f(u)dt+g(u)B(t)dt, \;
u(0)=u_0,
\end{align}
where $\sigma=1$, the random coefficient $B(t,x)=\mathcal V(x)\beta(t), x\in \mathbb R$ with a smooth real-valued potential $\mathcal V$ and a standard Brownian motion $\beta(t).$ Formally speaking, one may think that the driving process is $W(t)=\int_0^t B(r) dr$ in \eqref{SNLS}.
At this moment, our approach can not be used  to deal with \eqref{SNLS} directly since its temporal regularity is not enough to introduce the Bourgain space.  However,  one may follow our approach to show the well-posededness of \eqref{SNLS} driven by $W(t,x)=\beta^H(t)\mathcal V(x),$
 where $\beta^H$ is the centered fractional Brownian motion with the Hurst parameter $H>\frac 12$. 

 
In order to study the $\mathbb H^2$-regularity, we introduce the following definition $(\mathcal P_{\bs_0})$ which gives a kind of bilinear Strichartz estimate (see e.g., \cite{MR1691575}). 

\begin{df}
Let $\bs_0\in [0,1)$. We say $S(t)=\exp(\bi \Delta t)$ satisfies the property $(\mathcal P_{s_0})$ if for all the dyadic numbers $M,L$, and $v_0,w_0\in \mathbb H$ localized on the dyadic intervals of order $M,L$, i.e.,
\begin{align*}
\mathbb I_{M\le \sqrt{-\Delta}\le 2M}(v_0)=v_0, \; \mathbb I_{L\le \sqrt{-\Delta}\le 2L}(w_0)=w_0,
\end{align*}
it holds that 
\begin{align*}
\|S(t)v_0S(t)w_0\|_{L^2([0,1]\times \mathcal O)}\le C(\min(M,L))^{\bs_0} \|v_0\|\|w_0\|.
\end{align*}
\end{df}

Such definition was used by several authors in the context of the wave equations and the Schr\"odinger equations in deterministic case (see, e.g., \cite{BGT05,MR2428002}). For instance, it has been proved that $(\mathcal P_{\bs_0})$ holds with $\bs_0=(\frac 12)^+$ for $\mathbb S^3$ and $\bs_0=(\frac 34)^+$ for $\mathbb S^2\times \mathbb S^1.$ When $\mathcal O=\mathbb T^2$ and $\mathbb T^3$, $\bs_0=0^+$ and $(\frac 12)^+$ respectively.  For general manifolds with boundary, $\bs_0$ is shown to be $(\frac 34)^+$. 

Beyond the bilinear  Strichartz estimate, we also need to introduce the Bourgain spaces:
\begin{df}  
The Bourgain space $X^{\bs,b}(\mathbb R \times M)$ is the completion of $C_0^{\infty}(\mathbb R; \mathbb H^{\bs})$ under the norm 
\begin{align*}
\|u\|_{X^{s,b}(\mathbb R\times M)}^2
&=\sum_{i} \big\|(\tau+\lambda_k)^{\frac b2} (1+\lambda_k)^{\frac s 2} \widehat{ \<u,e_k\>}(\tau)\big\|^2_{L^2(\mathbb R_{\tau};\mathbb H)}\\
&= \|e^{-\bi t\Delta } u(t,\cdot)\|_{H^{\bs}(\mathbb R_t;\mathbb H^{\bs})}^2
\end{align*}
where $\widehat {(\cdot)}(\tau)$ denotes the Fourier transform with respect to time variable.
For any $T>0$, we denote by $X_{T}^{\bs,b}$ the space of restrictions of the elements of $X^{\bs,b}(\mathbb R \times M)$ endowed with the norm 
\begin{align*}
\|u\|_{X_T^{s,b}}=\inf \{\|\widetilde u\|_{X^{\bs,b}(\mathbb R \times M)}, \widetilde u \; \big|_{[0,T]\times \mathcal O}=u\}.
\end{align*}
\end{df}

The basic properties of the Bourgain space (see e.g. \cite{BGT05}), which will be used in the investigation $\mathbb H^{\bs}$-regularity, is summarized as follows. The Sobolev emebdding theorem holds, i.e., $X_{T}^{\bs,b} \hookrightarrow \mathcal C([0,T];\mathbb H^{\bs})$ and
$X_T^{0,\frac 14}  \hookrightarrow  L^4([0,T]; \mathbb H)$, as well as 
$X^{\bs_2,b_2}(\mathbb R \times M)   \hookrightarrow X^{\bs_1,b_1}_T (\mathbb R \times M) $ for $\bs_1\le \bs_2$ and $b_1\le b_2.$
In the following, we state the nonlinear estimate which is taken from \cite[Proposition 3.3 and Remark 3.4]{MR2428002}.

\begin{lm}\label{lm-non-est}
Under the condition of $(\mathcal P_{\bs_0}),$ let $\bs>\bs_0.$ There exists $(b,b')$ satisfying $0<b'<\frac 12<b, b+b'<1$, and $C>0$ such that for any $v,w,y \in X^{\bs,b}(\mathbb R\times \mathcal O),$
\begin{align*}
\|v \bar w y\|_{X^{\bs,-b'}(\mathbb R\times \mathcal O)}
&\le C\|v\|_{X^{\bs,b}(\mathbb R\times \mathcal O)}\|w\|_{X^{\bs,b}(\mathbb R\times \mathcal O)}\|y\|_{X^{\bs,b}(\mathbb R\times \mathcal O)}.
\end{align*}
In particular, when $v=w=y$, $s\ge 1,$  there exists $(b,b')\in \mathbb R^2$ satisfying $0<b'<\frac 12<b, b+b'<1$ and $C>0$ such that 
\begin{align*}
\||v|^2v\|_{X^{\bs,-b'}(\mathbb R\times \mathcal O)} \le C\|v\|_{X^{\bs,b}(\mathbb R\times \mathcal O)}\|v\|_{X^{1,b}(\mathbb R\times \mathcal O)}^2.
\end{align*}
\end{lm}

Now we are in a position to prove the $\mathbb H^{\bs}$-regularity,  $\bs\ge 2$, in the pathwise sense. 

\begin{prop}\label{3d-prop}
Let  $T>0$,  $\sigma=1$, $d\ge 2$, $\lambda=-1$, the condition $(\mathcal P_{\bs_0})$ hold and $u(0)\in \mathbb H^{\bs}, V\in \mathbb W^{\bs,\infty},  \bs>\bs_0.$
The solution of the Cauchy problem \eqref{SNLS} is locally well-posed in $\mathbb H^{\bs}, a.s.$
Moreover, if $\bs \ge 1,$ then the solution exists globally in $\mathcal C([0,T];\mathbb H^{\bs}), a.s.$ 
\end{prop}

\begin{proof}
The proof of the well-posedness of the additive noise case is similar to that of \cite[Theorem 3.4]{BD03} and thus is omitted. 
Below, we only present the proof of the multiplicative noise case. 
Let $T_1$ be a sufficient small number which will be determined later.
According to Duhamel's principle, we define a map $\Gamma$ from $X_{T_1}^{\bs,b}$ to itself by
$$\Gamma v= S(t)u(0)+\int_0^t S(t-s) f(v(s))ds+\int_0^t S(t-s) g(v(s))B(s)ds.$$
By using \cite[Proposition 2.11]{BGT05}, Lemma \ref{lm-non-est} and $X_T^{\bs,\frac 14}  \hookrightarrow  L^4([0,T]; \mathbb H^\bs)$, we achieve that 
\begin{align*}
\|S(t) u(0)\|_{X^{\bs,b}_{T}}&\le C\|u(0)\|_{\mathbb H^{\bs}},\\
\Big\|\int_0^t S(t-r) f(v(r)) dr\Big\|_{X^{\bs,b}_{T_1}}&\le CT_1^{1-b-b'} \|v\|_{X^{\bs,b}_{T_1}}^3,\\
\Big\|\int_0^tS(t-r) g(v(r))B(r)  dr\Big\|_{X^{\bs,b}_{T_1}}& \le C T_1^{1-b-b'}  \|v\|_{X^{\bs,b}_{T_1}} \|B\|_{X^{\bs,\frac 14}_{T_1}}. 
\end{align*}
Since $\beta(t)$ is $(\frac 12-\epsilon)$-H\"older continuous, $\epsilon\in (0,\frac 12)$, and $\mathcal V\in \mathbb W^{\bs,\infty}$, we have that $\|B\|_{X^{\bs,\frac 12-\epsilon}_{T}}<\infty, a.s.$
Thus, by taking $T_1\sim \min(\|u(0)\|_{\mathbb H^{\bs}}^{-\frac 2{1-b-b'}},\|B\|_{X^{\bs,\frac 14}_{T}}^{-\frac 1{1-b-b'}})$ and applying the fixed point theorem, it follows that the unique fixed point $u  \in {X^{\bs,b}_{T_1}}$ satisfies $\|u(t)\|_{X^{\bs,b}_{T_1}}\le C\|u(0)\|_{\mathbb H^{\bs}}.$

Next, we show the global existence of the solution when $\bs= 1.$ 
By the above arguments, it follows that there exists $T_1(\|u(0)\|_{\mathbb H^1}, \|B\|_{X^{1,\frac 14}_{T}})$ such that $\|u(t)\|_{X^{1,b}_{T_1}}\le C\|u(0)\|_{\mathbb H^{1}}, \;a.s.$
By applying the chain rule, H\"older's and Young's inequalities, one can obtain the finiteness  of mass and energy in any $p$-moment sense. 
This, together with $\lambda=-1$,  implies that 
\begin{align*}
\sup_{t\in [0,\tau]}\|u(t)\|_{\mathbb H^1} \le C_2(T,\|u(0)\|_{\mathbb H^1}), \; a.s.,
\end{align*}
where $C_2(T,\|u(0)\|_{\mathbb H^1})$ is a random variable independent of $\tau\in [0,T]$. As a consequence, we could extend the local solution to the global solution almost surely. 
When $\bs\ge 1$, applying Lemma \ref{lm-non-est}, it follows that there exists  $$T_1 \sim \min (C_2(T,\|u(0)\|_{\mathbb H^1})^{-\frac 2{1-b-b'}},\|B\|_{X^{1,\frac 14}_{T}}^{-\frac 2{1-b-b'}})$$ such that the solution exists and satisfies $
\|u\|_{X_{T_1}^{\bs,b}}\le C \|u(0)\|_{\mathbb H^{\bs}}.$ 
As $T_1$ depends only on $\|u(0)\|_{\mathbb H^{1}}$ and $\|B\|_{X^{1,\frac 14}_{T}}$, by elementary iterating process, the solution is global. 
\end{proof}

Due to loss of the moment estimate of $\|u(\cdot)\|_{\mathbb H^{\bs}}, \bs >1,$ or $\|u(\cdot)\|_{L^{\infty}}$, we only have the following result without any explicit decay rate on the tail estimate of \eqref{SNLS}.

\begin{cor}\label{3d-cor}
Under the condition of Proposition \ref{3d-prop}, it holds that 
\begin{align*}
\lim_{R_1\to\infty} \mathbb P(\sup_{r\in [0,T]} \|u(r)\|_{\mathbb H^{\bs}}\ge R_1)=0.
\end{align*}
Furthermore, when $\bs> \frac d2,$ it holds that
\begin{align*}
\lim_{R_1\to\infty} \mathbb P(\sup_{r\in [0,T]} \|u(r)\|_{L^{\infty}}\ge R_1)=0.
\end{align*}
\end{cor}

\section{Divergence phenomenon of explicit numerical methods for SNLSEs}\label{divergence}
Since the solution of SNLSE usually does not have an analytical expression, many researchers have been devoted to approximating it numerically and the structure-preserving numerical method has become one popular choice to discretize SNLSEs. 
In this section, we illustrate why the structure-preserving, truncated and  tamed strategies are needed to approximate SNLSEs via several explicit schemes.   We would like to mention that the divergence phenomenon of explicit  methods for SDEs with polynomial coefficients has been reported in \cite{MR2795791}. But less result is known in the infinite-dimensional case.

Below we prove the the divergence of  a class of explicit schemes for Eq.	\eqref{SNLS}, including 
the exponential Euler scheme
\begin{align}\label{exp-euler}
u_{n+1}^{(1)}=S(\delta t)\Big(u_{n}^{(1)}+\delta  tf(u_n^{(1)})+ g(u_n^{(1)})\delta _nW-\delta  t\alpha u_n^{(1)}\Big),
\end{align}
the linear-implicit modified mid-point scheme
\begin{align}\label{semi1}
u_{n+1}^{(2)}=S_{\delta t}u_{n}^{(2)}+ S_{\delta t} \delta  t f(u_n^{(2)})+ S_{\delta t}g(u_n^{(2)})\delta _nW-\delta  t \mathcal{T}_{\delta t} \alpha u_n^{(2)},
\end{align}
where $S_{\delta t}:=\frac {1+\frac12\bi\Delta\delta t }{1-\frac 12\Delta\bi \delta t}$ and $\mathcal T_{\delta t}:=\frac1{1-\frac 12\Delta\bi\delta t}$,
and the linear-implicit modified Euler scheme
\begin{align}\label{semi2}
u_{n+1}^{(3)}= \widehat T_{\delta t}u_{n}^{(3)} + S_{\delta t} \delta  t f(u_n^{(3)})+  S_{\delta t} g(u_n^{(3)})\delta _nW- \delta  t \widehat {\mathcal T}_{\delta t} \alpha u_n^{(3)},
\end{align}
where $\widehat {\mathcal T}_{\delta t}=\frac 1{1-\Delta\bi\delta t}$.
Here $n\le N, n\in \mathbb N, T=N\delta t.$
The following properties of these semi-groups 
\begin{align}\label{sem-gr}
\|S(\delta t)\|_{\mathcal L(\mathbb H,\mathbb H)}=\|S_{\delta t}\|_{\mathcal L(\mathbb H,\mathbb H)}=1, \ \|\widehat {\mathcal T}_{\delta t}\|_{\mathcal L(\mathbb H,\mathbb H)}\le C_0,  \ \|{\mathcal T}_{\delta t}\|_{\mathcal L(\mathbb H,\mathbb H)}\le C_0
\end{align}
for some $C_0>0$
 will be frequently used in this section.

\begin{lm}\label{div-num}
Let  $T>0,$ $d\ge 1,$ $\varPsi$ be an  $L^{4\sigma+2}$-valued random variable with $\sigma>0$ and $$\P\Big( \min\limits_{x\in \mathcal O}|g(\varPsi(x))|\neq 0\Big)>0.$$ Assume that $Q$ commutes with the Laplacian operator on $\mathcal O$ and $Q^{\frac 12}\in \mathcal L_2^{\bs}$ with $\bs=0$ in the additive noise case and $\bs>\frac d2$ in the multiplicative noise case. 
There exists a constant $c\in (1,\infty)$, $\epsilon\in (0,2\sigma)$ and  
a sequence of non-empty events $\Omega_N \in \FFF$, 
$N\in \N^+$, with $\mathbb P[\Omega_N]\ge c^{(-N)^c}$ and $u_N^{(l)}(\omega)\ge 2^{(2\sigma+1-\epsilon)^{(N-1)}}$ for $\epsilon>0$ small enough and $\omega \in \Omega_N$, $N \in \N^+$.
Moreover, the numerical approximations satisfy 
$\lim_{N\to \infty} \E[\|u_N^{(l)}\|^p]=\infty$, $l=1,2,3$,  for  $p\in [1,\infty)$.
\end{lm}
\begin{proof}
For simplicity, we only present the result of $u^{(1)}_N$ since the proof of the other cases is similar.  We first show the divergence in the multiplicative noise case.
Denote $m(\mathcal O)$  the Lebesgue measure of $\mathcal O.$
Notice that for the polynomial  $\delta tm(\mathcal O)^{-\sigma} \xi ^{2\sigma}-\delta t(\frac 1 2 \|\alpha\|_{L^{\infty}}+1)-1$,  we can always find  $C' \ge  \frac {2\delta t(\frac 1 2 \|\alpha\|_{L^{\infty}}+1)+2}{\delta t} $$ m(\mathcal O)^{\sigma}$ and  $\epsilon\in (0,2\sigma)$ 
 such that for $\xi \ge \max(C', (\frac {2m(\mathcal O)^{\sigma}}{\delta t})^{\epsilon})$, it holds that 
\begin{align*}
\delta t m(\mathcal O)^{-\sigma}\xi ^{2\sigma}-\delta t(\frac 1 2 \|\alpha\|_{L^{\infty}}+1)-1\ge \frac 12 \delta t m(\mathcal O)^{-\sigma} \xi^{2\sigma}\ge  \xi^{2\sigma-\epsilon}.
\end{align*}
For $C_0'=\frac 12 \|\alpha\|_{L^{\infty}}+1$, we define 
\begin{align*}
r_N:=\max\Bigg(2, C_0', \Big(\frac 2{\delta t}+C_0'\Big)^{\frac 1{2\sigma}}, C', (\frac {2m(\mathcal O)^{\sigma}}{\delta t})^{\epsilon} \Bigg)
\end{align*}
and the events $\Omega_N \in \FFF$, $N \in \N^+$ by
\begin{align*}
\Omega_N&:=\bigg\{\omega\in \Omega \big| \|\delta _nW\|_{\mathbb H^{\bs}}\in [\frac 12\delta t,\delta t],  n\in \mathbb Z_{N-1}/\{0\}, \|\delta W_0\| \ge K(r_N+K),\\
&\qquad\quad \min_{x\in\mathcal O} |\varPsi(x)| \ge \frac 1K, \|u_0\|+T\|f(u_0)-\alpha u_0\|\le K \bigg\},
\end{align*} 
where  $K>1$, $\bs=0$ for additive noise and $\bs>\frac d2$ for the multiplicative noise. 


Now, we use induction arguments to show that $\|u_n(\omega)\|\ge (r_N)^{(2\sigma+1-\epsilon)^{n-1}}$, $n\in \mathbb Z_N/\{0\}$ for given $\omega \in \Omega_N$. 
For $n=1$, the definition of $\Omega_N$ and \eqref{sem-gr} yields that 
\begin{align*}
 \|u_1^{(l)}(\omega)\|
  &\ge \|g(\varPsi)\delta W_0(\omega)\|-\|u_0(\omega)\|-\delta t \|f(u_0(\omega))-\alpha u_0(\omega)\|\\
 &\ge \frac 1K \|\delta W_0(\omega)\|-K\ge r_N\ge C\ge 1.
\end{align*}
 By inductions, assume that $\|u_n^{(l)}(\omega)\|\ge (r_N)^{(2\sigma+1-\epsilon)^{n-1}}$ holds for the first $n$ steps.
The unitarity  of $S(t)$ and H\"older's inequality  yield  that
\begin{align*}
&\|S(\delta t)(f(u)-\frac 12 \alpha)\|\\
&\ge \| |u|^{2\sigma}u\|-\frac 12\|\alpha\|_{L^\infty}\|u\|=\|u\|_{4\sigma+2}^{2\sigma+1}-\frac 12\|\alpha\|_{L^\infty}\|u\|\\
&\ge m(\mathcal O)^{-\sigma}\|u\|^{2\sigma+1}-\frac 12\|\alpha\|_{L^\infty}\|u\|.
\end{align*}
 To prove the estimate in the $(n+1)$-th step, by H\"older's inequality and the Sobolev embedding theorem, we have 
\begin{align*}
\|u^{(l)}_{n+1}(\omega)\|&
\ge \delta t\|f(u^{(l)}_{n}(\omega))-\alpha u_n(\omega)\|
-\|g (u^{(l)}_{n}(\omega))\delta _nW(\omega)\|-\|u_n^{(l)}(\omega)\|\\
&\ge  m(\mathcal O)^{-\sigma}\|u_n^{(l)}(\omega)\|^{2\sigma+1}\delta t-\frac 12\|\alpha\|_{L^\infty}\|u_n^{(l)}(\omega)\|\delta t-\delta t\|u^{(l)}_n(\omega)\|-\|u_n^{(l)}(\omega)\| \\
&\ge \|u^{(l)}_n(\omega)\|\Big(m(\mathcal O)^{-\sigma} \delta t\|u^{(l)}_n(\omega)\|^{2\sigma }-\delta t(\frac 1 2 \|\alpha\|_{L^{\infty}}+1) -1\Big).
\end{align*}
Therefore, by the definition of $r_N$, we get  
\begin{align*}
\|u^{(l)}_{n+1}(\omega)\|&\ge\|u^{(l)}_n(\omega)\|\Big(m(\mathcal O)^{-\sigma}\delta t\|u^{(l)}_n(\omega)\|^{2\sigma }-(\frac 1 2 \|\alpha\|_{L^{\infty}}+1)\delta t-1\Big)\\
&\ge \|u^{(l)}_n(\omega)\|^{2\sigma+1-\epsilon}.
\end{align*} 
Thus we can obtain 
\begin{align*}
\|u_N^{(l)}(\omega)\|\ge\|u^{(l)}_{N-1}(\omega)\|^ { 2\sigma+1-\epsilon}\ge (r_N)^{({2\sigma+1-\epsilon})^N}\ge 2^{({2\sigma+1-\epsilon})^N}
\end{align*}
for  $\omega \in \Omega_N$.

Next we give the lower bound of $\mathbb P(\Omega_N)$. 
Denote $\mathcal V_1:=\P\Big[\min\limits_{x\in \mathcal O}|\varPsi (x)|$ $\ge \frac 1K, \|\varPsi\|+T\|f(\varPsi)-\frac 12\alpha \varPsi\|\le K \Big]>0$.
Using the lower bound of the tail estimate of normal distribution $Z$, i.e., for $x\in [1,\infty)$, $\mathbb P(|Z|\ge x)\ge \frac 14 xe^{-x^2},$  it holds that  for $k\in \N^+$ and $Q^{\frac 12}e_k\neq 0,$
\begin{align}\label{exp-tai}
\P\Big[ T^{-\frac 12}\| W(t)\|\ge x \Big]
\ge\P\Big [T^{-\frac 12} |\beta_k(T)|\ge \|Q^{\frac 12}e_k\|^{-1}x\Big] 
\ge \frac  x{4\|Q^{\frac 12}e_k\| } \exp^{(-\|Q^{\frac 12}e_k\|^{-2}x^2)}.
\end{align}

Applying  \eqref{sem-gr}, \eqref{exp-tai}, the independent increments of the Wiener process and the property that $\frac {W(t)}{N^{\frac 12}}$ has the same distribution of $\delta W_0$, as well as $\mathbb P(|Z|\in [\frac 12 x, x])\ge \frac {x e^{-\frac 12x^2}}4$,  
we achieve that for some $k\in \mathbb N^+$ and $Q^{\frac 12}e_k\neq0,$
\begin{align*}
\mathbb P(\Omega_N)&\ge \mathcal V_1 \mathbb P\Big[\|\delta W_0\|\ge K(r_N+K)\Big]
\Big(\mathbb P\Big[\|\delta W_0\|_{\mathbb H^{\bs}}\in [ \frac 12 \delta t,\delta t]\Big]\Big)^{N-1}\\
&\ge \mathcal V_1 \frac {K(r_N+K)}{4\|Q^{\frac 12}e_k\| \delta t^{\frac 12}}\exp\Big(-\frac {K^2(r_N+K)^2\|Q^{\frac 12}e_k\|^2}{\delta t}-T\|Q^{\frac 12}e_k\|^2_{\mathbb H^{\bs}}\Big)\\
&\quad \times \Big(\frac {\delta t^{\frac 12}}{4 \|Q^{\frac 12}e_k\|_{\mathbb H^{\bs}}}\Big)^N\\
&\ge \frac {\mathcal V_1 (\delta t)^{\frac {N-1}2}e^{-T\|Q^{\frac 12}e_k\|^2_{\mathbb H^{\bs}}} } {(4^{N+1}\|Q^{\frac12}e_k\|_{\mathbb H^{\bs}})^{N}\|Q^{\frac12}e_k\|}
\exp\Big(-\frac {K^2(r_N+K)^2\|Q^{\frac 12}e_k\|^2_{\mathbb H^{\bs}}}{\delta t}\Big).
\end{align*}
Since $N\log N\le N^c$, for  $c \in (1,\infty)$, there exists 
a constant $c \in (1,\infty)$ such that 
\begin{align*}
\mathbb P[\Omega_N]\ge c^{-(N)^c}
\end{align*}
for $N\in \N^+$.
The lower bound of $u^N_n$ on $\Omega_N$ and the lower bound of $\mathbb P(\Omega_N)$ yield that for any $p\ge 1$,
\begin{align*}
\lim_{N\to \infty}\E [\|u^{(l)}_N\|^p]&\ge\lim_{N\to \infty} \Big(\mathbb P(\Omega_N) (r_N)^{p(2\sigma+1-\epsilon)^{N-1}}\Big)\\
&\ge \lim_{N\to \infty} \Big(\mathbb P(\Omega_N) 2^{p(2\sigma+1-\epsilon)^{N-1}}\Big)\\
&\ge \lim_{N\to \infty} \Big(c^{(-N^c)} 2^{p(2\sigma+1-\epsilon)^{N-1}}\Big)=\infty.
\end{align*}
For the additive noise case, the proof is similar. The main difference is that    the definition of $\gamma_N$ 
will rely on the polynomial $\delta tm(\mathcal O)^{-\sigma} \xi ^{2\sigma+1}-\delta t \frac 1 2 \|\alpha\|_{L^{\infty}}\xi-\delta t-\xi.$ We omit the tedious and analogous procedures for simplicity.
\end{proof}

\begin{rk}
The proposed approach can also be extended to prove the divergence of explicit schemes for general stochastic Schr\"odinger nonlinear equations whose drift and diffusion coefficients have different growing speeds, i.e.,
there exist $C_1 >0, C_2\ge 1,$ $\beta>1, \beta>\alpha\ge 0$,
\begin{align*}
\max (\|f(v)\|,\|g(v)\|)\ge C_1 \|v\|^{\beta}, \; \text{and}\; \min (\|f(v)\|,\|g(v)\|)\le C_2\|v\|^{\alpha},
\end{align*}
for all $\|x\|\ge C$ with a large enough number $C>0.$
\end{rk}

\section{Explicit numerical methods}
\label{sec-scheme}

Roughly speaking, there exist two popular approaches to treat the divergence problem of numerical methods for SDEs. One can use some stable implicit discretizations such that the $p$-moment of numerical solutions is finite (see, e.g., \cite{MR1949404}). This approach often requires the solvability of underlying implicit schemes. Besides, the overall error estimate of the implementable algorithm involved with Newton's iteration or Picard's iteration is not clear for non-monotone SPDEs in general.  Another approach to overcome the divergence issue is constructing different variants of the explicit schemes, such as the truncated/tamed strategy and the adaptive mesh method (see, e.g., \cite{MR3129758,MR3829168}). 
In this section, we focus on the second approach and will construct and analyze strongly  convergent explicit schemes for SNLSEs.

\subsection{Construction}
Via the Lie--Trotter splitting (see e.g. \cite{CHS21}), \eqref{SNLS} can be split into several stochastic Hamiltonian systems (see e.g. \cite{Liu13,CHLZ19}). Following this structure-preserving strategy, we first propose an explicit splitting scheme which reads 
\begin{align}\label{spl}
u_{n+1}^M&=\Phi_{S,n}^M(\delta t, \Phi_{D,n}^{M}(\delta t, u_n^M)),
\end{align}
where $\Phi_{D,n}^{M},\Phi_{S,n}^{M}, t\in [t_n,t_{n+1}]$ are the phase flows of the following subsystems 
\begin{align}\label{sub-det}
&dv_{D,n}^M(t)=\bi \Delta  v_{D,n}^M (t)dt, \; v_{D,n}^M(t_n)= u^M_{n}, u_{S,n}(t_n)=v_{D,n}^M(t_n), \\\label{sub-sto}
&du_{S,n}^M(t)=P^M\Big(\bi \lambda |u_{S,n}^M(t_n)|^{2\sigma}u_{S,n}^M(t)-\frac 12\alpha u_{S,n}^M(t)\Big)dt+ P^M g(u_{S,n}^M(t)) dW(t),
\end{align}  
respectively. 
Here $n\le N$, $N\delta t=T$, $u_0^M=P^M \varPsi$ and $P^M$ is the spectral Galerkin projection operator with $M\in \mathbb N^+$.
We would like to remark that \eqref{sub-sto} is a linear system and thus it is analytically solvable. One can easily construct different explicit splitting numerical schemes via this splitting idea. 
For instance, splitting \eqref{SNLS} into three parts yields the following scheme
\begin{align*}
u_{n+1}^M=\widehat \Phi_{S,n}^M(\delta t, \Phi_{f,n}^M(\delta t,  \Phi_{D,n}^{M}(\delta t, u_n^M))),
\end{align*} 
where $\widehat \Phi_{S,n}^M,  \Phi_{f,n}^M, \Phi_{D,n}^{M}$ correspond to the subsystems 
\begin{align*}
&d x(t) = \bi P^M \Delta x(t)dt,\\
&d y(t) = \bi P^M(\lambda |y(t_n)|^{2\sigma} y(t)-\frac 12 \alpha y(t))dt, \\
&d z(t) = \bi P^M g(z(t))dW(t).
\end{align*}

Another popular way to construct stable explicit scheme is using  truncated/tamed method.  For example, applying the truncated strategy to the standard exponential Euler schemes leads to  the nonlinearity-truncated fully discrete exponential Euler scheme,
{\small
\begin{align*}
u_{n+1}^M=S(\delta t) P^M u_n+\mathbb I_{\{\|u_{n}^M\|_{\mathbb H^{\kappa}}\le R \}}S(\delta t) P^M (f(u_n^M)-\frac 12 \alpha u_n^M)\delta t +S(\delta t) P^M g(u_{n}^M)\delta _nW,
\end{align*}
}
where the indicator function is defined on $\Omega$. One may change the truncated norm $\mathbb H^{\kappa}$, $\kappa\in \mathbb N^+$, and set the scale of  $R$ to find the suitable relationship of $R, \delta t, M$ such that this scheme enjoys the desirable higher regularity. 

When $g(\xi)=\bi \xi,$ the proposed splitting method \eqref{spl} is a kind of the nonlinearity-truncated scheme due to the property of the spectral Galerkin method (see e.g. \cite[Subsection 4.1]{CHLZ19}) that 
\begin{align}\label{mass-decay}
\|u_{n}^M\|\le \|u_0^M\|, \; a.s.
\end{align}
In the following,  we will focus on \eqref{spl} and present the detailed steps for proving the strong convergence of explicit numerical methods of SNLSEs. Following similar steps, the convergence analysis of other explicit splitting numerical methods and the nonlinearity-truncated fully discrete exponential Euler scheme could be established. For simplicity, we omitted the tedious calculations for these schemes. 

\subsection{Stability and Higher regularity estimate}
The first step to prove the strong convergence of \eqref{spl} lies on the stability and higher regularity estimate. Below, we show that there always exists a relationship between $M$ and $\delta t$ such that \eqref{spl} inherits the higher regularity estimate of the exact solution.

\begin{prop}\label{Sta1}
Let $p\in \mathbb N^+$ and $T>0$. Suppose that either Assumption \ref{add} or Assumption \ref{mul} holds with $\bs=1$.
The scheme \eqref{spl} is unconditionally stable, i.e.,
\begin{align*}
\E\Big[\sup_{n\le N}\|u_n^M\|^{2p}\Big]<\infty,
\end{align*}
and enjoys $\mathbb H^1$-regularity, i.e.,
\begin{align*}
\E\Big[\sup_{n\le N}\|u_n^M\|_{\mathbb H^1}^{2p}\Big]<\infty
\end{align*}
under the condition that 
\begin{align}\label{con-spa-tim}
\lambda_M^{2+2\kappa_1\sigma}\delta t \sim O(1), \; \text{for some} \; \kappa_1>\frac d2.
\end{align}
\end{prop}

\begin{proof}
We first present the details of the proof for the multiplicative noise case.
When $g(\xi)=\bi \xi$, since \eqref{sub-det} is a Hamiltonian system, and the mass of \eqref{sub-sto}  decays,  \eqref{mass-decay} holds, which implies the unconditional stability. 
In order to show the $\mathbb H^1$-estimate, it suffices to prove the a priori estimate of the energy functional. This will be done  via a bootstrap argument as follows. 
For simplicity, we only present the details for $p=1$.
Applying the It\^o formula  to $H(u_{S,n}^M(t_{n+1}))$ and using the integration by parts, we get 
{\small
\begin{align*}
&H(u_{S,n}^M(t_{n+1}))\\
&=H(u_n^M)+\int_{t_n}^{t_{n+1}} \<\bi \Delta v_{D,n}^M(s),  \bi f(v_{D,n}^M(s))\>ds+\int_{t_n}^{t_{n+1}}\<-\Delta u_{S,n}^M(s), \bi \lambda |u_{S,n}^M(t_n)|^{2\sigma}u_{S,n}^M(s)\> ds\\
&+\int_{t_n}^{t_{n+1}}\< -\Delta u_{S,n}^M(s), -\frac 12 \alpha u_{S,n}^M(s)\>ds+\int_{t_n}^{t_{n+1}}\<-\Delta u_{S,n}^M(s), g(u^M_{S,n}(s))dW(s)\> \\
&+\int_{t_n}^{t_{n+1}}\frac 12 \sum_{i} \<\nabla g(u^M_{S,n}(s)) Q^{\frac 12}e_i, \nabla   g(u^M_{S,n}(s)) Q^{\frac 12}e_i\> ds\\
&+\int_{t_n}^{t_{n+1}}\<\lambda \bi f(u_{S,n}^M(s)), P^M\Big(\bi \lambda |u_{S,n}^M(t_n)|^{2\sigma}u_{S,n}^M(s)  -\frac 12 \alpha u_{S,n}^M(s)\Big) \>\\
&+\int_{t_n}^{t_{n+1}}\<\lambda \bi f(u_{S,n}^M(s)), P^M g(u^M_{S,n}(s))dW(s)\>\\
&+\int_{t_n}^{t_{n+1}}\frac 12 \sum_{i}\<\bi\lambda |u_{S,n}^M(s)|^{2\sigma} P^M g(u^M_{S,n}(s))Q^{\frac 12}e_i, P^M g(u^M_{S,n}(s))Q^{\frac 12}e_i\>ds\\
&+\int_{t_n}^{t_{n+1}}\frac 12 \sum_{i}\<\bi\lambda |u_{S,n}^M(s)|^{2\sigma-2} 2Re\Big(u_{S,n}^M(s) P^M g(u^M_{S,n}(s))Q^{\frac 12}e_i \Big)  u_{S,n}^M(s),\\
&\quad P^M g(u^M_{S,n}(s))Q^{\frac 12}e_i\>ds.
\end{align*}
}
Using the fact that 
\begin{align}\label{smo-semi}
\|(S(\delta t)-I)w\|\le C\delta t^{\min(\frac {\bs'} 2, 1)} \|w\|_{\mathbb H^{\bs'}}, \; \bs'\in \mathbb N^+,
\end{align}
and applying the Gagliardo--Nirenberg interpolation inequality, $$\|w\|_{L^{4\sigma+2}} \le C\|\nabla w\|^{\frac {\sigma d}{2\sigma+1}} \|w\|^{\frac {2\sigma+1-d}{2\sigma+1}}, d\le 2,$$ as well as the inverse inequality $\|w\|_{\mathbb H^{\bs_1}}\le C\lambda_{M}^{\frac {\bs_1} 2}\|w\|, $ $\forall \; w\in P^M \mathbb H, \;\bs_1\in \mathbb N^+$,
one can show that for $s\in [t_n,t_{n+1}],$
\begin{align}\label{con-est-vdn}
\|v_{D,n}^M(s)-u_n^M\|&\le  C\delta t \lambda_{M}\|u_n^M\| \\\label{con-est-usn}
\|u_{S,n}^M(s)- u_n^M\|&\le C\delta t^{\max(\frac \bs 2,1)}\lambda_{M}^{\frac \bs 2} \|u_n^M\|
+C\Big\|\int_{t_n}^s u_{S,n}^M(r)dW(r)\Big\|\\\nonumber
&+C\int_{t_n}^s \lambda_M^{\frac {\sigma d}2} (1+ \|u_{S,n}^M(r)\|^{2\sigma+1} +\|u_{S,n}^M(t_n)\|^{2\sigma+1})dr. 
\end{align}
According to \eqref{con-est-vdn} and \eqref{con-est-usn}, using the Gagliardo--Nirenberg interpolation inequality, Holder's and Young's inequalities, and the inverse inequality, we have that 
\begin{align*}
&\Big|\int_{t_n}^{t_{n+1}} \<\bi \Delta v_{D,n}^M(s),  \bi f(v_{D,n}^M(s))\> ds+\int_{t_n}^{t_{n+1}}\<-\Delta u_{S,n}^M(s), \bi \lambda |u_{S,n}^M(t_n)|^{2\sigma}u_{S,n}^M(s)\> ds\Big|\\
&\le C \int_{t_n}^{t_{n+1}} \lambda_M \Big(1+\|v_{D,n}^M(s)\|_{L^{4\sigma+2}}^{2\sigma+1}+\|u_{S,n}^M(t_n)\|_{L^{4\sigma+2}}^{2\sigma+1}\Big) \|v_{D,n}^M(s)-u_{S,n}^M(s)\|ds\\
&+C \int_{t_n}^{t_{n+1}} \lambda_M^{1+\sigma \kappa_1} \Big(1+\|v_{D,n}^M(s)\|^{2\sigma+1}+\|u_{S,n}^M(t_n)\|^{2\sigma+1}\Big) \|v_{D,n}^M(s)-u_{S,n}^M(s)\|ds\\
&\le C \lambda_M^{1+\kappa_1\sigma}  \int_{t_n}^{t_{n+1}}\Big(1+\|v_{D,n}^M(s)\|^{2\sigma+1}+\|u_{S,n}^M(t_n)\|^{2\sigma+1}\Big) \Big(\delta t \lambda_{M} \|u_n^M\|\\
&\quad + \Big\|\int_{t_n}^s u_{S,n}^M(r)dW(r)\Big\|+\lambda_M^{\frac {\sigma d}2}\int_{t_n}^s  (1+ \|u_{S,n}^M(r)\|^{2\sigma+1} +\|u_{S,n}^M(t_n)\|^{2\sigma+1})dr\Big).
\end{align*}
Taking expectation on the equation of $H(u_{S,n}^M(t_{n+1}))$, and using H\"older's and Young's inequalities, as well as the Gagliardo--Nirenberg interpolation inequality and 
the property of Galerkin projection operator,
\begin{align}\label{smo-spe}
\|(I-P^M) w\|\le C \lambda_M^{-\frac \bs 2}\|w\|_{\mathbb H^{\bs}},
\end{align}
we obtain that for $\lambda_M^{2+2\kappa_1 \sigma}\delta t\sim \mathcal O(1), \kappa_1>\frac d 2$, 
\begin{align*}
&\E \Big[H(u_{S,n}^M(t_{n+1}))\Big]\\
&\le \E \Big[H(u_n^M)\Big]+C\delta t +C\int_{t_n}^{t_{n+1}} H(u_{S,n}^M(s)) ds\\
&+C\int_{t_n}^{t_{n+1}} \|(I-P^M) u_{S,n}^M(s)\alpha \|_{L^{2\sigma+2}}^{2\sigma+2} ds
\\
&+C\int_{t_n}^{t_{n+1}}\sum_{i} \|(I-P^M) u_{S,n}^M(s)Q^{\frac 12}e_i \|_{L^{2\sigma+2}}^{2\sigma+2} ds\\
&+C\int_{t_n}^{t_{n+1}} \Big|\< P^M  f(u_{S,n}^M(s)),\Big(|u_{S,n}^M(s)|^{2\sigma}-|u_{S,n}^M(t_n)|^{2\sigma}\Big)u_{S,n}(s)\>\Big| ds\\
&\le  \E \Big[H(u_n^M)\Big]+C\delta t +C\int_{t_n}^{t_{n+1}} H(u_{S,n}^M(t_n)) ds\\
&+C\int_{t_n}^{t_{n+1}} \E \Big[ \|\nabla (I-P^M)u^M_{S,n}(s)\|^{\sigma d} \|(I-P^M)u^M_{S,n}(s)\|^{2\sigma+2-\sigma d} \Big]ds\\
&\le \E \Big[H(u_n^M)\Big]+C\delta t +C\int_{t_n}^{t_{n+1}}\E\Big[ H(u_{S,n}^M(t_n)) \Big]ds,
\end{align*}
where we use the condition  $\sigma d\le 2$ in  Assumption \eqref{mul} in the last step.
By Gronwall's inequality, we have $ \sup\limits_{n\le N}\E \Big[H(u_n^M)\Big] <\infty.$
Similarly, applying It\^o formula  to $H^{p}(u_{S,n}^M(t_{n+1}))$ and repeating the above procedures, we have that   
\begin{align}\label{gron-wall}
\sup_{n\le N}\sup_{t\in [t_n,t_{n+1}]}\E\Big[H^p(u_{S,n}^M(t))\Big]<\infty.
\end{align}
One can introduce the global auxiliary processes $v_{D}^M$ and $u_{S}^M$ defined  by 
\begin{align}\label{glo-aux-def}
&v_{D,M}(t)=v_{D,M,n}(t), \; \text{if} \;\;  t\in [t_n,t_{n+1}), \; \text{and}\;\; v_{D,M}(t_{n+1})=u_{n+1},\\\nonumber
&u_S(t)=u_{S,n}(t), \; \text{if} \;\;  t\in (t_n,t_{n+1}], \text{and} \;\; u_{S}(t_{n})= u_{S,n-1}(t_{n}).
\end{align}
It can be seen that $v_{D}^M$ and $u_{S}^M$ are $\mathcal F_t$-adapted and continuous piece-wisely. Following the arguments in the proof of \cite[Corollary 2.1 and Lemma 3.3]{CHLZ19}, it is not hard to show that  
\begin{align}\label{glo-aux-pri}
\E \Big[\sup_{s\in [0,T]}\|v_{D,M}(s)\|_{\mathbb H^{1}}^{2p}\Big]+\E \Big[\sup_{s\in [0,T]}\|u_{S}(s)\|_{\mathbb H^{1}}^{2p}\Big]<\infty.
\end{align}
As a consequence,  it follows that $\E \Big[\sup\limits_{n\le N} \|u_n\|_{\mathbb H^1}^{2p}]<\infty.$
For the additive noise case, the proof is similar and simpler since when applying the It\^o formula, the second derivative of $H(u_{S,n}^M(t))$ can be bounded by the energy itself. 
\end{proof}

\begin{rk}\label{rk-con}
It can be seen that 
the condition $\lambda_M^{2+2\kappa_1\sigma}\delta t\sim \mathcal O(1),$ 
is just for simplicity and could be improved. 
One may add a tamed term or truncated term into the stochastic nonlinear subsystem \eqref{sub-sto} to improve this restrictive condition. Our numerical analysis of the explicit numerical methods is also applicable to other spatial discretizations, like the finite difference methods and finite element methods, as long as the corresponding inverse holds. 
\end{rk}

Next, we show that the scheme \eqref{spl} could inherit the higher regularity estimate in $\mathbb H^{\bs}, \bs \ge 2$, of the exact solution, whose proof is given in the appendix.

\begin{prop}\label{cor-high}
Let $p\in \mathbb N^+$, $T>0$ and \eqref{con-spa-tim} hold. Suppose that either Assumption \ref{add} or Assumption \ref{mul} holds with $\bs\ge 2$. 
Then the numerical solution of  \eqref{spl} satisfies that for $d=1$,
\begin{align}\label{sup-pri-hs-1d}
\E \Big[ \sup_{n\le N} \|u_n^M\|_{\mathbb H^{\bs}}^{2p}\Big]<\infty,
\end{align}
and for $d=2$,
\begin{align}\label{sup-pri-hs-2d}
\E \Big[ \sup_{n\le N} \widetilde U(u_n^M)\Big]<\infty.
\end{align}
\end{prop}

One could also derive the regularity of the numerical scheme when $\sigma$  is not an integer and show that under the condition of Corollary \ref{cor-high} with $\sigma\ge \frac 12$,  \eqref{sup-pri-hs-2d} holds. As a consequence, the following tail estimates hold.

\begin{cor}\label{tail-num}
Let the condition of Proposition \ref{cor-high}. For a large enough $R_1\ge 1$ and $p_1\in \mathbb N^+$, it holds that for $d=1,$
\begin{align}\label{tail-est-un}
\mathbb P\Big( \sup\limits_{n\le N} \|u_n^M\|_{\mathbb H^{\bs}} \ge R_1 \Big)\le C_1 R_1^{-p_1},
\end{align}
and that for $d=2$,
\begin{align}\nonumber
&\mathbb P\Big( \sup\limits_{n\le N} \|u_n^M\|_{\mathbb H^{1}} \ge R_1 \Big)\le C_1 R_1^{-p_1},\\\label{tail-est-un-2d}
&\mathbb P\Big( \sup\limits_{n\le N} \|u_n^M\|_{L^{\infty}} \ge R_1 \Big)\le C_1 \log^{-p_1}(R_1).
\end{align}
\end{cor}

To end this section, we present the exponential integrability of the energy functional for 1D stochastic cubic Schr\"odinger equation whose proof is based on \cite[Lemma 3.1]{CHL16b}, \eqref{mass-decay} and the  Gagliardo--Nirenberg interpolation inequality $\|w\|_{L^{\infty}}\le C_0 \|\nabla w\|^{\frac 12} \|w\|^{\frac 12}+C_0' \|w\|.$ We omit the details of this proof for simplicity.

\begin{prop}\label{exp-1d}
Let $T>0$, $d=1$, \eqref{con-spa-tim} and Assumption \ref{mul} hold with $\bs=1$ and $\sigma=1$. Then for  a large enough $R_1\ge 1$ and $p_1\in \mathbb N^+$, the scheme   \eqref{spl} satisfies 
\begin{align*}
\mathbb P\Big(\sup\limits_{n\le N} \|u_n^M\|_{\mathbb H^{1}} \ge R_1\Big)\le C_1 \exp(-\eta_1 R_1^2),
\end{align*}
and 
\begin{align}\label{exp-l}
\mathbb P\Big(\sup\limits_{n\le N} \|u_n^M\|_{L^{\infty}} \ge R_1\Big)\le C_1 \exp(-\eta_1 R_1^4).
\end{align}
\end{prop}

\section{Error decomposition and strong convergence}

In this section, we are in a position to present our methodology to analyze the convergence of explicit numerical methods and then to establish strong convergence analysis of \eqref{spl}. 

\subsection{Error decomposition}\label{sub-err-dec}
By iteratively rewriting \eqref{spl} into its integral form, we have that 
\begin{align*}
u_{n+1}^M
&=S(t_{n+1})P^M u_0^M+\int_{0}^{t_{n+1}} S_{N,n+1}^M(s) \bi \lambda |u_S^M(t_n)|^{2\sigma} u_S^M(s) ds\\
&-\int_{0}^{t_{n+1}} S^M_{N,n+1}(s) \frac 12 \alpha u^M_{S}(s) ds
+\int_{0}^{t_{n+1}} S^M_{N,n+1}(s) g(u_{S}^M(s)) dW(s),
\end{align*}
where $S_{N,n+1}^M(s)$ is defined by 
\begin{align*}
S^M_{N,n+1}(s)= \sum_{j=0}^{n-1} I_{[t_j,t_{j+1}]}(s) S(t_n-t_j) P^M+I_{[t_n,t_{n+1}]}(s) P^M.
\end{align*}
Thanks to $\mathbb H^{\bs_1}$-regularity of $u(t)$ in section 2, it follows that 
\begin{align}\label{strong-spe}
\|u(t)-P^M u(t)\|^{2p} \le C \lambda_M^{-\frac {\bs_1}2}\|u(t)\|_{\mathbb H^{\bs_1}}.
\end{align}
Therefore, in order to show the strong convergence of \eqref{spl}, it suffices to estimate the error $P^M u(t_n)-u_m^N$. To simplify the presentation, we only present the detailed steps of the strong convergence analysis of \eqref{spl} in the multiplicative noise case. 
In the following, we denote $[s]$ the integer part of $\frac {s}{\delta t}, s\in [0,T]$ and $[s]_{\delta t}=[s]\delta t.$

\subsubsection{Error decomposition in additive noise case}

The strong convergence analysis of \eqref{spl} in the additive noise case could be established by directly investigating the mild formulas of $P^Mu(t_n)$ and $u_n^M$ which is simpler than the multiplicative noise case. 
Define a sequence of subsets of $\Omega,$
\begin{align}\label{trun-seq}
&\Omega_{R_1}^{n+1}:=\Big\{\sup\limits_{s\in [0,(n+1)\delta t]} \Big(\|u(s)\|_{\mathbb H^{1}}+ \|u_{[s]}^M\|_{\mathbb H^{1}}\Big) \le R_1 \Big\}, \; \text{if } d=1, \; \text {and }\\\nonumber
&\Omega_{R_1}^{n+1}:=\Big\{\sup\limits_{s\in [0,(n+1)\delta t]} \Big(
(1+\|u(s)\|_{\mathbb H^1})(1+\log(1+\|u(s)\|_{\mathbb H^2}^2))
\\\nonumber
&\qquad\qquad\quad+(1+\|u_{[s]}^M\|_{\mathbb H^1})(1+\log(1+\|u_{[s]}^M\|_{\mathbb H^2}^2))
\Big) \le R_1 \Big\}, \; \text{if } d=2.
\end{align}
It can be seen that $\Omega_{R_1}^{n}$ is increasing w.r.t. $R_1$ and decreasing w.r.t. $n.$
We omit the details in the additive noise case and only present its error representation formula  here.   

\begin{prop}\label{prop-aux}
Let $T>0, p\in \mathbb N^+$ and Assumption \ref{add} hold with $\bs>\frac d2$. 
Then it holds that 
\begin{align}\label{err-for-add}
&\|u_{n+1}^M-P^M u(t_{n+1})\|_{L^p(\Omega_{R_1}^{n+1} ;\mathbb H)}\\\nonumber
& 
\le C(\delta t^{\frac 12} +\lambda_M^{-\frac {\bs} 2}) \sup_{s\in [0,t_{n+1}]}\Big(1+\|u(s)\|^{2\sigma+1}_{L^{(2\sigma+1)p}(\Omega_{R_1}^{n+1};\mathbb H^{\bs})}+\|u_S^M(s)\|^{2\sigma+1}_{L^{(2\sigma+1)p}(\Omega_{R_1}^{n+1};\mathbb H^{\bs})} \Big)\\\nonumber
&+C\int_{0}^{t_{n+1}} R_1^{2\sigma'}\|u_{[r]}-P^M u([r]_{\delta t})\|_{L^p(\Omega_{R_1}^{[r]};\mathbb H)}ds,
\end{align}
where $\sigma'=\frac \sigma 2$ if $d=1$, and $\sigma'=\sigma$ if $d=2.$

\end{prop}

\subsubsection{Error decomposition in multiplicative noise case}

To eliminate the interaction effects of the unbounded operator and the multiplicative noise, we introduce the exponential transform 
$v_n(t)=S(-(t-t_n)) u(t), t\in [t_n,t_{n+1}]$ which satisfies the following SODE,
\begin{align*}
dv_n(t)&=S(-(t-t_n))  f(S(t-t_n)v_n(t))dt -\frac 12S(-(t-t_n)) \alpha (S(t-t_n) v_n(t))dt\\
&\quad +S(-(t-t_n)) g(S(t-t_n)v_n(t))dW(t),
\end{align*}
Via the above equation and \eqref{sub-sto}, we can transform the error estimate of SPDEs into that of infinite-dimensional SDEs.
It can be also seen that
\begin{align*}
dP^M v_n(t)&=S(-(t-t_n)) P^M f(S(t-t_n)P^Mv_n(t)+(I-P^M)u(t))dt \\
&-\frac 12S(-(t-t_n)) P^M \alpha (S(t-t_n) v_n(t)+(I-P^M)u(t))dt\\
&\quad +S(-(t-t_n))P^M g(S(t-t_n)v_n(t)+(I-P^M)u(t))dW(t).
\end{align*}
To propose the error decomposition, we also need the sequence of subsets $\Omega_{R_1}^{n}$ defined in \eqref{trun-seq} which is increasing w.r.t. $R_1$ and decreasing w.r.t. $n.$

\begin{prop}\label{prop-err-dec}
Let $T>0, p\in \mathbb N^+$ and Assumption \ref{add} hold with $\bs>\frac d2$.
 It holds that for $\bs_1\in \mathbb N^+$,
\begin{align}\label{err-for-mul}
&\|P^M u(t_{n+1})-u_{n+1}^M\|_{L^{2p}(\Omega^{n+1}_{R_1};\mathbb H)}^2\\\nonumber
&\le 
C(1+R_1^{2\sigma'})\delta t+C\int_{0}^{t_{n+1}}(1+R_1^{2\sigma'}) \|P^Mu([r]_{\delta t})-u_{[r]}^M\|^2_{L^{2p}(\Omega^{[r]}_{R_1};\mathbb H)} dr \\\nonumber
&+C(1+R_1^{2\sigma'})\lambda_M^{-\bs}\int_{0}^{t_{n+1}}\Big(1+\|u(r)\|_{L^{2p}(\Omega^{n+1}_{R_1};\mathbb H^{\bs})}^2\Big)ds\\\nonumber
&+C(\lambda_M^{-\bs_1}+\delta t^{\min(2,\bs_1)})
\sup_{r\in [0,t_{n+1}]}   \Big(\|u(r)\|^2_{L^{2p}(\Omega;\mathbb H^{\bs_1})}+\|u_{S}^M(r)\|_{L^{2p}(\Omega;\mathbb H^{\bs_1})}^2\Big)\\\nonumber
&+ C (\lambda_M^{-\frac {\bs_1}2}+\delta t^{\min(1,\frac {\bs_1}2)})\sup_{r\in [0,t_{n+1}]}   \Big(\|u(r)\|_{L^{2p}(\Omega;\mathbb H^{\bs_1})}+\|u_{S}^M(r)\|_{L^{2p}(\Omega;\mathbb H^{\bs_1})}\Big) \\\nonumber
&\times \Big(\int_0^{t_{n+1}}  \|P^Mu([r]_{\delta t})-u_{[r]}^M\|^2_{L^{2p}((\Omega^{[r]}_{R_1})^c;\mathbb H)} dr\Big)^{\frac 12}, \nonumber
\end{align}
where $\sigma'=\frac \sigma 2$ if $d=1$, and $\sigma'=\sigma$ if $d=2.$
\end{prop}

\begin{proof}
For convenience, let us define a global auxiliary process $v(t)$ by 
$v(t)=v_n(t)$ for $t\in [t_n,t_{n+1})$ and $v(t_{n+1})=S(-\delta t) u(t_{n+1})$. 
For $t\in [t_n,t_{n+1}),$ applying the It\^o formula to $\|P^M v(t)-S(-\delta t) u_S^M (t)\|^2$  and using the fact $\<w,\bi w\>=0$, $w\in \mathbb H$, we obtain 
{\small
\begin{align*}\nonumber
&\|P^M v(t)-P^M S(-\delta t)u_S^M(t)\|^2\\
&=\|P^M u(t_n)-\lim_{t\to t_n}  S(-\delta t) u_S^M(t) \|^2\\\nonumber
&+2\int_{t_n}^{t}\<P^M v(r)-S(-\delta t) u_S^M(r), S(t_n-r) f(S(r-t_n) v(r))-S(-\delta t)\bi \lambda |u_S^M(t_n)|^{2\sigma} u_S^M(r)\>dr\\\nonumber
&+2\int_{t_n}^{t}\<P^M v(r)-S(-\delta t)u_S^M (r),(S(t_n-r) g(S(r-t_n) v(r))-S(-\delta t) g(u_S^M(r)))dW(r) \>\\\nonumber
&=\|P^M u(t_n)-u_n^M\|^2\\\nonumber
&+2\int_{t_n}^{t}\<P^M v(r)-S(-\delta t) u_S^M(r), S(t_n-r) f(S(r-t_n) v(r))-S(-\delta t) f(S(\delta t)v(r))\>dr\\
&+2\int_{t_n}^{t}\<P^M v(r)-S(-\delta t) u_S^M(r), S(-\delta t) f(S(\delta t)v(r))-S(-\delta t) \bi \lambda |u_S^M(t_n)|^{2\sigma} u_S^M(r) \>dr\\
&+2\int_{t_n}^{t}\<P^M v(r)-S(-\delta t) u_S^M(r),(S(t_n-r) g(S(r-t_n) v(r))-S(-\delta t) g(S(\delta t)v(r)))dW(r) \>\\\nonumber
&+2\int_{t_n}^{t}\<P^M v(r)-S(-\delta t) u_S^M(r),(S(-\delta t) g(S(\delta t)(I-P^M)v(r)))dW(r) \>\\
&=:\|P^M u(t_n)-u_n^M\|^2+Err_1^n(t)+Err_2^n(t)+Err_3^n(t)+Err_4^n(t).
\end{align*}
}
Notice that for $t\in [t_n,t_{n+1}]$,
\begin{align}\label{property-v}
&\|v(t)-u(t)\|= \|(S(t-t_n)-I)u(t)\|\le C\delta t^{\min(\frac \bs 2,1)}\|u(t)\|_{\mathbb H^{\bs}},\\\nonumber
&v(t)-v(t_n)=\int_{t_n}^{{t}}\Big(S(t_n-r) f(S(r-t_n) v(r))-\frac 12S(t_n-r)  \alpha S(r-t_n) v(r)\Big)dr\\\nonumber
&\qquad\qquad\qquad +\int_{t_n}^t S(t_n-r)  g(S(r-t_n) v(r)) dW(r),
\end{align}
and 
\begin{align}\label{property-us}
u_S^M(t)-u_S^M(t_n)&=\int_{t_n}^{{t}} P^M \Big(\bi \lambda |u_S^M(t_n)|^{2\sigma}u_S^M(r)-\frac 12 \alpha u_S^M(r)\Big)dr\\\nonumber
&+\int_{t_n}^t P^M g(u_S^M(r)) dW(r).
\end{align}
According to \eqref{property-v} and the fact that $\mathbb H^{\bs}, \bs>\frac d2$ forms an algebra, using the unitarity of $S(\cdot)$, \eqref{smo-semi}, and the Sobolev embedding theorem, we have that 
\begin{align*}
Err_1^n(t) \le C \int_{t_n}^t \|P^M v(r)-S(-\delta t) u_S^M(r)\| \delta t^{\min(1,\frac \bs 2)}(1+\|u(r)\|_{\mathbb H^{\bs}}^{2\sigma+1})dr.
\end{align*}
Similarly, using \eqref{property-us}, \eqref{smo-spe}, Holder's and Young's inequalities, as well as Propositions \ref{1d-prop},  \ref{d=2-h2} and  \ref{Sta1}, it follows that 
\begin{align*}
&Err_2^n(t)\\
 &\le C \int_{t_n}^t \|P^M v(r)-S(-\delta t) u_S^M(r)\|^2 \big(1+\|P^M v(r)\|_{L^{\infty}}^{2\sigma}+\|u_S^M(t_n)\|_{L^{\infty}}^{2\sigma}\big)dr\\
&+ C \int_{t_n}^t \|P^M v(r)-S(-\delta t) u_S^M(r)\| \Big(\|u_S^M(t_n)\|_{L^{\infty}}^{2\sigma} +\|v(r)\|_{L^{\infty}}^{2\sigma} \Big) \|(I-P^M) v(r)\|dr\\
&+ C \int_{t_n}^t \|P^M v(r)-S(-\delta t) u_S^M(r)\| \Big(\|u_S^M(t_n)\|_{L^{\infty}}^{2\sigma} +\|v(r)\|_{L^{\infty}}^{2\sigma} \Big) \|u_S^M(r)-u_S^M(t_n)\|dr\\
&\le C\big(1+R_1^{2\sigma’} \big) \Big[ \int_{t_n}^t \|P^M v(t_n)-S(-\delta t) u_S^M(t_n)\|^2 dr + \int_{t_n}^t\Big(\|v(r)\|_{\mathbb H^{\bs}}^2 \lambda_M^{-\bs} \\
&+\delta t\|u_S^M(t_n)\|_{\mathbb H^{1}}^2+\int_{t_n}^r\|u_S^M(s)\|^{4\sigma+2}_{\mathbb H^1} ds +\|\int_{t_n}^r g(u_S^M(s))dW(s)\|^2 \Big)dr \Big]\\
&+C\int_{t_n}^t \Big|\int_{t_n}^{s} \Big(1+\|v(r)\|_{L^{4\sigma+2}}^{2\sigma+1}+\|u_S^M(r)\|_{L^{4\sigma+2}}^{2\sigma+1}\Big)dr\Big|^2dr\\
&+C\int_{t_n}^t \Big\| \int_{t_n}^r P^M \Big(
S(t_n-s)g(S(s-t_n)v(r))-g(u_S^M(s)) 
 \Big)dW(s)\Big\|^2dr.
\end{align*}
Here the stochastic integral term will be dealt with similar to $Err_3^n$ and $Err_4^n$ and thus its estimate is omitted.

Using  \eqref{property-v} and \eqref{property-us}, as well as Burkholder's inequality, it follows that for $\bs_1\in \mathbb N^+,$
{\small
\begin{align*}
&\|\sum_{j=0}^{n}Err_3^{j}(t_{j+1})+Err_4^{j}(t_{j+1})\|_{L^{p}(\Omega_{R_1}^{n+1};\mathbb R)}\\
&\le 
C\Big(\int_{0}^{t_{n+1}}\Big(\E\Big[ \Big\|P^M u([r]_{\delta t})- u_{[r]}^M \Big\|^p\Big[\|(I-P^M)v(r)Q^{\frac 12}\|_{\mathcal L_2^0}+\|((S(r-t_n)-S(-\delta t))\\
&\qquad  v(r))Q^{\frac 12}\|_{\mathcal L_2^0}+\|(S(r-t_n)-S(-\delta t)) v(r)Q^{\frac 12}\|_{\mathcal L_2^0}\Big]^p \Big)^{\frac 2p}dr\Big)^{\frac 12}\\ 
&+C\Big(\int_{0}^{t_{n+1}} \Big(\E\Big[  \Big\|P^M v(r)-P^M v([r])+S(-\delta t) (u_S^M([r])-u_S^M(r))\Big\|^p \\ 
&\qquad \Big[\|(I-P^M)v(r)Q^{\frac 12}\|_{\mathcal L_2^0}+\|(S(r-t_n)-S(-\delta t)) v(r)Q^{\frac 12}\|_{\mathcal L_2^0}\\
&\qquad+\|(S(r-t_n)-S(-\delta t)) v(r)Q^{\frac 12}\|_{\mathcal L_2^0}\Big] ^p \Big)^{\frac 2p}dr\Big)^{\frac 12}\\
&\le \Big(\sum_{j\le n} \|P^M u(t_j)-u_{j}^M\|_{L^{2p}(\Omega;\mathbb H)}^2 \delta t\Big)^{\frac 12}  (\lambda_M^{-\frac {\bs_1} 2}+\delta t^{\min(\frac {\bs_1} 2,1)})
\sup_{r\in [0,t_{n+1}]}   \|u(r)\|_{L^{2p}(\Omega;\mathbb H^{\bs_1})} \\
 &+ \Big(\int_0^{t_{n+1}} \|P^M v(r)-P^M v([r])+S(-\delta t) (u_S^M([r])-u_S^M(r))\|_{L^{2p}(\Omega;\mathbb H)}^2 dr\Big)^{\frac 12}  \\
 &\times (\lambda_M^{-\frac {\bs_1} 2}+\delta t^{\min(\frac {\bs_1} 2,1)})
\sup_{r\in [0,t_{n+1}]}   \|u(r)\|_{L^{2p}(\Omega;\mathbb H^{\bs_1})}.
\end{align*} 
}
 By applying \eqref{property-v}, \eqref{property-us}, the Sobolev embedding theorem, Propositions \ref{1d-prop},  \ref{d=2-h2} and  \ref{Sta1}, it follows that for $r\in [t_j,t_{j+1}],$
 \begin{align*}
 &\|P^M v(r)-P^M v([r])+S(-\delta t) (u_S^M([r])-u_S^M(r))\|_{L^{2p}(\Omega;\mathbb H)}\\
 &\le 
 \Big\|\int_{t_n}^{{r}}\Big(S(t_n-s) f(S(s-t_n) v(s))-\frac 12S(t_n-s)  \alpha S(s-t_n) v(s)\Big)ds\Big\|_{L^{2p}(\Omega;\mathbb H)}\\\nonumber
 &+ \Big\|\int_{t_n}^{{r}}\Big(\bi \lambda |u_S^M(t_n)|^{2\sigma}u_S^M(s)-\frac 12 \alpha u_S^M(s)\Big) ds\Big\|_{L^{2p}(\Omega;\mathbb H)}\\
& + \Big \|\int_{t_n}^r P^M (S(t_n-s)  g(S(s-t_n) v(s))-g(u_S^M(s)))dW(s) \Big\|_{L^{2p}(\Omega;\mathbb H)}\\
&\le C\delta t +C\delta t^{\frac 12}\|P^Mu(t_j)-u_j^M\|_{L^{2p}(\Omega;\mathbb H)}
\\
&+C\delta t^{\frac 12}(\lambda_M^{-\frac {\bs_1} 2}+\delta t^{\min(\frac {\bs_1} 2,1)})
\sup_{r\in [0,t_{n+1}]}   \Big(\|u(r)\|_{L^{2p}(\Omega;\mathbb H^{\bs_1})}+\|u_{S}^M(r)\|_{L^{2p}(\Omega;\mathbb H^{\bs_1})}\Big).
 \end{align*} 
By expanding $\|P^M v(t)-S(-\delta t) u_S^M (t)\|^2$ at the initial error, then taking $L^p(\Omega_{R_1}^{n+1};\mathbb R)$ and using Propositions \ref{1d-prop},  \ref{d=2-h2} and  \ref{Sta1}, we have that 
\begin{align*}
&\|P^M v(t_{n+1})-S(-\delta t) u_S^M (t_{n+1})\|^2_{L^{2p}(\Omega_{R_1}^{n+1};\mathbb H)}\\
&\le \sum_{j\le n} \Big(\|Err_1^j(t_{j+1})\|_{L^{p}(\Omega_{R_1}^{n+1};\mathbb H)}+\|Err_2^j(t_{j+1})\|_{L^{p}(\Omega_{R_1}^{n+1};\mathbb H)}\Big)\\
&+\|\sum_{j\le n}Err_3^{j}(t_{j+1})+Err_4^{j}(t_{j+1})\|_{L^{p}(\Omega_{R_1}^{n+1};\mathbb R)}.
\end{align*}
Summarizing up the estimates of $Err_1^n-Err_4^n$, and using Young's inequality,
we obtain \eqref{err-for-mul}.
\end{proof}

\subsection{Strong convergence}
In this part, 
we provide a new type of stochastic Gronwall's inequality and combine it with Proposition  \ref{prop-aux} and Proposition \ref{prop-err-dec} to show the strong convergence of the proposed scheme.

\subsubsection{Truncated stochastic Gronwall's inequality}

The following lemma is based on a series of truncated subsets of $\Omega$ which has its own interests and could be used to study the strong convergence problem for general non-monotone SDEs and SPDEs. 

\begin{lm}\label{rough-gronwall}
Let $\{\Omega_{R_1}^{n}\}_{n\le N}$ be a sequence of subsets of $\Omega$ which is decreasing with respect to $n$ and increasing with respect to $R_1\ge 1$,  $\sigma_1\ge 0$ and $\epsilon_N\ge 0.$
Assume that a random positive sequence $\{a_{n}\}_{n\le N}$ in a Banach space $\mathbb E$ satisfies 
\begin{align*}
\sup_{n\le N}\|a_{n}\|_{L^{q_0}(\Omega;\mathbb E)}\le C(q_0), \; \forall  \; q_0\in \mathbb N^+,
\end{align*}
and that for some $C'\ge 0,$ 
{\small
\begin{align*}
\|a_{n+1}\|_{L^p(\Omega_{R_1}^{n+1};\mathbb E)} &\le C(1+R_1^{2\sigma_1}) \sum_{k=0}^n \|a_k\|_{L^p(\Omega_{R_1}^{k};\mathbb E)} \delta t + \epsilon_N
+ C'(\sum_{k=0}^{n} \|a_k\|_{L^p((\Omega_{R_1}^k)^c;\mathbb E)} \delta t)^{\frac 12}\epsilon_{N}^{\frac 12}. 
\end{align*} 
}
Then it holds that for some $l> 1,$
\begin{align*}
\|a_{n}\|_{L^p(\Omega;\mathbb E)} 
&\le \exp(C(T)R_1^{2\sigma_1}) \Big(\epsilon_N +C\epsilon_N^{\frac 12}  \sup_{k\le N}(\mathbb P((\Omega_{R_1}^k)^c))^{\frac 1{pl}}\Big)+C \sup_{k\le N} (\mathbb P((\Omega_{R_1}^k)^c))^{\frac 1{pl}}.
\end{align*}
\end{lm}

\begin{proof}
By using the uniform boundedness of $a_n$ in $L^p(\Omega;\mathbb E)$, H\"older's inequality and the Chebyshev inequality, we have that for $l> 1,$
\begin{align*}
\|a_{n+1}\|_{L^p(\Omega_{R_1}^{n+1};\mathbb E)} \le  C (1+R_1^{2\sigma_1}) \sum_{k\le n} \|a_k\|_{L^p(\Omega_{R_1}^k;\mathbb E)} \delta t + C\epsilon_N^{\frac 12}\sup_{k\le N}   (\mathbb P((\Omega_{R_1}^k)^c))^{\frac 1{pl}}+\epsilon_N.
\end{align*} 
By using the discrete Gronwall's inequality, we achieve that 
\begin{align*}
\|a_{n}\|_{L^p(\Omega_{R_1}^n;\mathbb E)} \le \exp(C(T)R_1^{2\sigma_1})\Big(\epsilon_N +C\epsilon_N^{\frac 12} \sup_{k\le N}   (\mathbb P((\Omega_{R_1}^{k})^c))^{\frac 1{pl}} \Big).
\end{align*}
Using the Chebyshev inequality and the boundedness of  $a_n$ in $L^p(\Omega;\mathbb E)$
yield that 
\begin{align*}
\|a_{n}\|_{L^p(\Omega;\mathbb E)} 
&\le \exp(C(T)R_1^{2\sigma_1}) \Big(\epsilon_N +C\epsilon_N^{\frac 12}  (\mathbb P((\Omega_{R_1}^N)^c))^{\frac 1{pl}}\Big)+C(\mathbb P((\Omega_{R_1}^N)^c))^{\frac 1{pl}}.
\end{align*}
\end{proof}

\begin{cor}\label{str-con}
Under the assumption of Lemma \ref{rough-gronwall}, if $ \lim\limits_{R_1\to \infty} \lim\limits_{N\to\infty}\mathbb P((\Omega_{R_1}^N)^c)$ $=0,$ and $\lim\limits_{N\to \infty} \epsilon_N =0.$
Then $\sup\limits_{n\le N}\|a_{n}\|_{L^p(\Omega;\mathbb E)}$ is convergent to $0.$
\end{cor}

\begin{proof}
By applying Lemma \ref{rough-gronwall} and letting $\epsilon_N \to 0$ firstly, 
we have that  
\begin{align*}
\|a_{n}\|_{L^p(\Omega;\mathbb E)} 
&\le  C \lim_{N\to\infty }\sup_{k\le N}(\mathbb P((\Omega_{R_1}^k)^c))^{\frac 1{pl}}\le  \lim_{R_1\to \infty}\lim_{N\to \infty} C (\mathbb P((\Omega_{R_1}^N)^c))^{\frac 1{pl}}.
\end{align*}
Taking $R_1$ goes to $\infty$, we complete the proof.
\end{proof}

When one is interested in the convergence rate of the numerical schemes, the following bootstrap type estimates will play a key role.

\begin{prop}\label{poly-gron}
Under the assumption of Corollary \ref{str-con}, suppose that 
$\mathbb P((\Omega_{R_1}^N)^c)$ $ \le C_1 R_1^{-p_1}$ for a large enough $p_1\in \mathbb N^+.$ Then it holds that for $\kappa>1,$ 
\begin{align*}
\|a_n\|_{L^p(\Omega;\mathbb E)}\le 2CC_1^{\frac 1{pl}}  \Big( \frac {\log((\epsilon_N+\epsilon_N^{\frac 12})^{-1})} {\kappa C(T)} \Big)^{-\frac {p_1}{2\sigma_1 p l}}.
\end{align*}
\end{prop}

\begin{proof}
For convenience, we assume that $C (\mathbb P((\Omega_{R_1}^N)^c))^{\frac 1{pl}} \le 1$. 
According to Lemma \ref{rough-gronwall}, letting $\exp(R_1^{2\sigma_1}C(T))(\epsilon_N+\epsilon_N^{\frac 12})\le CC_1^{\frac 1{pl}} R_1^{-\frac {p_1} {pl}}$, 
we have that for $\kappa>0$
\begin{align*} 
\|a_{n}\|_{L^p(\Omega;\mathbb H)} 
&\le 2CC_1^{\frac 1{pl}} R_1^{-\frac {p_1} {pl}}.
\end{align*}
Taking $R_1= [\frac {\log((\epsilon_N+\epsilon_N^{\frac 12})^{-1})} {\kappa C(T)}]^{\frac 1{2\sigma_1}}$ leads to the desired result.
\end{proof}

\begin{prop}\label{exp-con}
Under the assumption of Corollary \ref{str-con}, suppose that 
$\mathbb P((\Omega_{R_1}^N)^c)$ $\le C_1\exp(-\eta R_1^{2\sigma_2})$ with $\eta>0, \sigma_2>0$. Then it holds that 
\begin{align}\label{gen-exp-lm}
\|a_n\|_{L^p(\Omega;\mathbb E)}\le  2C C_1^{\frac 1{pl}} \exp \Big(-\frac  \eta {pl}  \big(\frac {\log((\epsilon_N+\epsilon_N^{\frac 12})^{-1})}{C(T)} \big)^{\frac {\sigma_2}{\sigma_1}} \Big).
\end{align}
In particular, when $\sigma_1< \sigma_2,$  it holds that for any $\gamma_1 \in (0,1),$
\begin{align}\label{exp-gron-lm}
\|a_n\|_{L^p(\Omega;\mathbb E)}\le C \epsilon_N ^{1-\gamma_1}.
\end{align}
\end{prop}

\begin{proof}
For convenience, we assume that $C (\mathbb P((\Omega_{R_1}^N)^c))^{\frac 1{pl}} \le 1$. 
 Applying Lemma \ref{rough-gronwall} and  letting $\exp(R_1^{2\sigma_1}C(T))(\epsilon_N+\epsilon_N^{\frac 12})\le CC_1^{\frac 1 {pl}}\exp(-\frac  \eta {pl} R_1^{2\sigma_2})$,   we have that 
\begin{align*}
\|a_n\|_{L^p(\Omega;\mathbb H)} \le  2C C_1^{\frac 1 {pl}}\exp(-\frac \eta {pl}  R_1^{2\sigma_2}).
\end{align*}
Taking $R_1=[\frac {(\epsilon_N+\epsilon_N^{\frac 12})^{-1}}{\kappa C(T)}]^{\frac 1{2\sigma_1}}$ yields that 
\begin{align*}
\|a_n\|_{L^p(\Omega;\mathbb E)} \le 2C C_1^{\frac 1{pl}} \exp \Big(-\frac  \eta {pl}  \big(\frac {\log((\epsilon_N+\epsilon_N^{\frac 12})^{-1})}{\kappa C(T)} \big)^{\frac {\sigma_2}{\sigma_1}} \Big).
\end{align*}
It can be seen that for $\sigma_1= \sigma_2$, 
$
\|a_{n}\|_{L^p(\Omega;\mathbb E)} 
\le 2CC_1^{\frac 1{pl}} (\epsilon_N+\epsilon_N^{\frac 12})^{\frac \eta {pl \kappa C(T)}}.
$
When $\sigma_1<\sigma_2$,  we  use the following bootstrap arguments to improve the convergence rate.
Applying Lemma \ref{rough-gronwall} yields that for a large enough $R_1>0$ such that 
\begin{align*}
\|a_{n}\|_{L^p(\Omega;\mathbb E)} 
&\le \exp(R_1^{2\sigma_1}C(T)) \Big(\epsilon_N+C\epsilon_N^{\frac 12} (\mathbb P(\Omega_{R_1}^c))^{\frac 1{pl}} \Big)+C(\mathbb P(\Omega_{R_1}^c))^{\frac 1{pl}}\\
&\le \exp(R_1^{2\sigma_1}C(T))  \Big( \epsilon_N+\epsilon_N^{\frac 12} \Big)
+ CC_1^{\frac 1 {p l}} \exp\Big(-\frac {\eta}{p l} R_1^{2\sigma_2} \Big).
\end{align*}
Taking $\frac {\eta}{pl}R_1^{2\sigma_2}=\log\Big((\epsilon_N+\epsilon_N^{\frac 12})^{-1}\Big),$
we have that 
\begin{align*}
\|a_{n}\|_{L^p(\Omega;\mathbb E)} 
&\le \big(1+\exp(R_1^{2\sigma_1}C(T))\big) \epsilon_N^{\frac 12}\\
&\le \Big(1+\exp\Big( (\frac {pl}{\eta})^{\frac {\sigma_1}{\sigma_2}} \log^{\frac {\sigma_1}{\sigma_2}}\big((\epsilon_N+\epsilon_N^{\frac 12})^{-1}\big)  C(T)\Big) \Big)(\epsilon_N+\epsilon_N^{\frac 12}).
\end{align*}
Noticing that $
\exp\Big( (\frac {pl}{\eta})^{\frac {\sigma_1}{\sigma_2}} \log^{\frac {\sigma_1}{\sigma_2}}\big((\epsilon_N+\epsilon_N^{\frac 12})^{-1}\big)  C(T)\Big)\le  (\epsilon_N+\epsilon_N^{\frac 12})^{-\gamma}$, $\forall \; \gamma\in (0,1),$
it follows that 
\begin{align*}
\|a_{n}\|_{L^p(\Omega;\mathbb E)} \le (\epsilon_N+\epsilon_N^{\frac 12})^{1-\gamma}.
\end{align*}
Now applying again Lemma \ref{rough-gronwall} and Young's inequality, we get for sufficient small $\gamma>0,$
\begin{align*}
\|a_{n}\|_{L^p(\Omega;\mathbb E)} 
&\le \exp(R_2^{2\sigma_1}C(T)) \Big(\epsilon_N +\epsilon_N^{\frac 12}   (\epsilon_N+\epsilon_N^{\frac 12})^{1-\gamma} \Big)+C(\mathbb P((\Omega_{R_2}^N)^c))^{\frac 1{pl}}\\
&\le C\exp(R_2^{2\sigma_1}C(T)) \Big(\epsilon_N +\epsilon_N^{1-\gamma} ) +C(\mathbb P((\Omega_{R_2}^N)^c))^{\frac 1{pl}}.
\end{align*}
By repeating the above procedures and taking $\frac {\eta}{pl}R_2^{2\sigma_2}=\log\Big((\epsilon_N+\epsilon_N^{1-\gamma})^{-1}\Big),$ we complete the proof.
\end{proof}

At the end of this part, we present one more estimate which will be used in studying 2D SNLSEs. 

\begin{prop}\label{prop-log-gron}
Under the assumption of Corollary \ref{str-con}, suppose that $\mathbb P((\Omega^N_{R_1})^c)$ $\le C_1 \log^{-p_1}(R_1)$ for a large enough $p_1\ge 1,$
Then it holds that 
\begin{align}\label{log-gron}
\|a_n\|_{L^{p}(\Omega;\mathbb E)}\le 2C C_1^{\frac 1{pl}} (2\sigma_1)^{\frac {p_1}{pl}} \log^{-\frac {p_1}{pl} }\Big(\log \big((\epsilon_N+\epsilon^{\frac 12} \big)^{-1})-\log(\kappa C(T))\Big).
\end{align}
\end{prop}

\begin{proof}
Without loss of generality, we assume that $C  (\mathbb P((\Omega_{R_1}^N)^c))^{\frac 1{pl}} \le 1.$ 
Applying Lemma \ref{rough-gronwall} and  letting $\exp(R_1^{2\sigma_1}C(T))(\epsilon_N+\epsilon_N^{\frac 12})\le CC_1^{\frac 1 {pl}}\log^{-\frac {p_1}{pl}}(R_1)$,   we have that 
\begin{align*}
\|a_n\|_{L^p(\Omega;\mathbb E)} \le  2C C_1^{\frac 1 {pl}}\log^{-\frac {p_1}{pl}}(R_1).
\end{align*}
Taking $R_1=[\frac {\log((\epsilon_N+\epsilon_N^{\frac 12} )^{-1})}{\kappa C(T)}]^{\frac 1{2\sigma_1}}$ yields that 
\begin{align*}
\|a_n\|_{L^p(\Omega;\mathbb E)} &\le 2C C_1^{\frac 1{pl}} \log^{-\frac {p_1}{pl} }\Big(
\Big[\frac {\log((\epsilon_N+\epsilon_N^{\frac 12} )^{-1})}{\kappa C(T)}\Big]^{\frac 1{2\sigma_1}}\Big)\\
&\le 2C C_1^{\frac 1{pl}} (2\sigma_1)^{\frac {p_1}{pl}} \log^{-\frac {p_1}{pl} }\Big(\log \big((\epsilon_N+\epsilon_N^{\frac 12} )\big)^{-1}-\log(\kappa C(T))\Big).
\end{align*}

\end{proof}

\subsubsection{Strong convergence}
By combining the truncated stochastic Gronwall's inequality (Lemma \ref{rough-gronwall}, Propositions \ref{poly-gron}-\ref{exp-con} and Corollary \ref{prop-log-gron}), and the regularity estimates and tail estimates of the exact and numerical solutions in Section 2 and Section 4, we are in a position to show several convergence properties of the proposed schemes for SNLSEs.

\begin{tm}\label{tm-1d}
Let $T>0, p\in \mathbb N^+$, $d=1$ and \eqref{con-spa-tim} hold. Suppose that Assumption \ref{add} or \ref{mul} holds with $\bs \ge 1$. Then 
the  scheme \eqref{spl}  satisfies that
\begin{align}\label{1d-err-add}
&\sup_{n\le N}\|u(t_n)-u_n\|_{L^{2p}(\Omega;\mathbb H)}\le C(T,Q,\varPsi,\lambda,\sigma,p)  \log^{-\kappa_1}\Big((\delta t^{\frac 12}+\lambda_M^{-\frac \bs 2})^{-1}\Big), \forall \; \kappa_1 \ge 1.
\end{align}
Furthermore, for $g(\xi)=\bi \xi$ and $\sigma=1,$ it holds that
\begin{align}\label{1d-err-mul}
&\sup_{n\le N}\|u(t_n)-u_n\|_{L^{2p}(\Omega;\mathbb H)}\le C(T,Q,\varPsi,\lambda,p)\Big(\delta t^{\frac 12}+\lambda_M^{-\frac \bs 2}\Big)^{1-\gamma}, \; \forall \; \gamma\in (0,1).
\end{align}
\end{tm}

\begin{proof}
In the additive noise case, applying Lemma \ref{rough-gronwall} and Proposition \ref{poly-gron}  to \eqref{err-for-add} 
with $a_n=\|u(t_n)-u_n\|$, $C'=0$, $\sigma_1=\sigma’$, $\epsilon_N\sim  O(\delta t^{\frac 12}+\lambda_M^{-\frac \bs 2})$, and using the tail estimates in Corollary \ref{1d-cor}   yield \eqref{1d-err-add}.
In the multiplicative noise case, applying Lemma \ref{rough-gronwall} and Proposition \ref{poly-gron} to \eqref{err-for-mul}, taking $a_n=\|u(t_n)-u_n\|^2$, $\bs=\bs_1$ and $\epsilon_N\sim  O(\delta t^{\frac 1 2}+\lambda_M^{-\frac \bs 2}), \sigma_1=\sigma'$, using Corollarys \ref{1d-cor} and \ref{tail-num}, Propositions \ref{1d-prop} and \ref{cor-high}, lead to \eqref{1d-err-add}.

When $g(\xi)=\bi \xi$ and $\sigma=1,$  applying Lemma \ref{rough-gronwall} and \eqref{exp-gron-lm} with $a_n=\|u(t_n)-u_n\|^2$,
$\sigma_1=\sigma',$ $\bs=\bs_1$ and $\epsilon_N\sim O(\delta t^{\frac 1 2}+\lambda_M^{-\frac \bs 2})$, using Corollarys \ref{1d-cor} and \ref{tail-num}, Propositions \ref{1d-prop},  \ref{cor-high} and \ref{exp-1d}, yield \eqref{1d-err-mul}.
\end{proof}

Compared with the existing strongly convergent results for 1D SNLSEs \cite{CH17,CHLZ19}, our strong analysis only requires $\mathbb H^{1}$-regularity to achieve the strong convergence order $\frac 12$ in time. The advantage of the implicit splitting Crank--Nicolson scheme in \cite{CH17, CHLZ19} is its unconditional stability in energy space. By applying the truncated stochastic Gronwall's inequality and Proposition \ref{exp-con} to its error decomposition one can obtain \eqref{1d-err-mul} if $d=1, \sigma<2$ and the convergence rate $O(\lambda_M^{\bs}+\delta t)^{\frac \eta {2pl \kappa C(T)}}$ if $d=1,\sigma=2.$ When $T$, the scale of $W$, or $\|\varPsi\|$ is sufficiently small,  via the large deviation principle of SNLSEs and our current approach, it seems possible to prove the strong convergence estimate \eqref{1d-err-mul} for general $\sigma$.  This will be investigated in the future.

\begin{tm}\label{tm-2d}
Let $T>0, p\in \mathbb N^+$, $d=2$ and \eqref{con-spa-tim} hold. Suppose that Assumption \ref{add} or \ref{mul} holds with $\bs \ge 2$. 
The scheme \eqref{spl} satisfies that for $\kappa_1\ge 1,$
\begin{align}\label{con-2d}
\sup_{n\le N}\|u(t_n)-u_n\|_{L^{2p}(\Omega;\mathbb H)}\le C(T,Q,\varPsi,\lambda,p)\log^{-\kappa_1}\Big(\log\big(\delta t^{\frac 12}+\lambda_M^{-\frac 1 2}\big)^{-1}\Big).
\end{align}
\end{tm}

\begin{proof}
In the additive noise case, applying Lemma \ref{rough-gronwall} and Proposition \ref{prop-log-gron}  to \eqref{err-for-add} 
with $a_n=\|u(t_n)-u_n\|$, $\bs=1,$ $C'=0$, $\sigma_1=\sigma'$, $\epsilon_N\sim R_1^{2\sigma'} O(\delta t^{\frac 12}+\lambda_M^{-\frac 1 2})$, and using the tail estimates  \eqref{2d-exp-lin} and \eqref{tail-est-un-2d}  yield \eqref{con-2d}.
In the multiplicative noise case, applying Lemma \ref{rough-gronwall} and Proposition \ref{poly-gron} to \eqref{prop-log-gron}, taking $a_n=\|u(t_n)-u_n\|^2$, $\bs_1=\bs=1$ and $\epsilon_N\sim R_1^{2\sigma'} O(\delta t^{\frac 1 2}+\lambda_M^{-\frac 1 2}), \sigma_1=\sigma'$, using \eqref{2d-exp-lin} and \eqref{tail-est-un-2d}, Propositions \ref{d=2-h2}, \ref{Sta1} and \ref{cor-high}, lead to \eqref{1d-err-add}.
\end{proof}

At last, we present the convergence of \eqref{spl} for Eq. \eqref{SCSE} with random coefficients on a compact Riemannian manifold in $d\ge 2$ without boundary.

\begin{tm}\label{3d-tm}
Let $d\ge 2, \lambda=-1,\sigma=1$, $p\in \mathbb N^+,$ $\bs > \frac d2, \bs\in \mathbb N^+$,  $\varPsi\in \mathbb H^{\bs}$, $\mathcal V(\cdot)\in W^{\bs,\infty}$ and \eqref{con-spa-tim} hold. 
Suppose that the scheme \eqref{spl} satisfies  $\lim\limits_{R_2 \to \infty}\lim\limits_{N\to\infty} \mathbb P((\widetilde \Omega_{R_2}^N)^c) =0,$ where 
\begin{align*}
\widetilde \Omega_{R_2}^{n}:=\Big\{\sup\limits_{s\in [0,t_{n}]} \|u_{[s]}^M\|_{\mathbb H^{\bs}} \le R_2, 
\Big\},\; R_2\in \mathbb R^+, \; n\le N.
\end{align*}
Then the numerical solution is strongly convergent to the solution of \eqref{SCSE}, i.e.,
\begin{align*}
\lim_{N\to \infty}\sup_{n\le N}\|u(t_n)-u_n\|_{L^{2p}(\Omega;\mathbb H)}=0.
\end{align*} 
\end{tm}

\begin{proof}
The proof is similar to that of Theorem \ref{tm-1d}. For the sake of simplicity,
we only present a sketch of this proof since the solution of Eq. \eqref{SCSE} with random coefficient is more regular than that of Eq. \eqref{SNLS}.
At first step, one could follow the steps in Section 4 to show that $\E \Big[\sup\limits_{t\in[0,T]}\|u(t)\|_{\mathbb H^1}^{2p}\Big]+\E \Big[\sup\limits_{n\le N} \|u_n^M\|_{\mathbb H^1}^{2p}\Big]<\infty$.
Second, similar to Proposition \ref{prop-aux}, by using the mild formulation of the numerical solution, one obtain the following error decomposition,
\begin{align*}
&\big\|P^Mu(t_{n+1})-u_{n+1}^M\big\|_{L^{2p}(\Omega_{R_1}^{n+1};\mathbb H)}\\\nonumber
&\le 
C(1+R_1^{2}) \int_{0}^{t_{n+1}}\|P^M u([s]_{\delta t})- \widetilde u_{[s]}\|_{L^{2p}(\Omega_{R_1}^{[s]};\mathbb H)} ds+
 C(1+R_1^{6})(\delta t^{\min(1,\frac \bs 2)}+\lambda_M^{-\frac {\bs}2}),
\end{align*}
where $\Omega_{R_1}^n:=\widetilde \Omega_{R_1}^n\cap\{\sup\limits_{s\in [0,t_n]} \|u(s)\|_{\mathbb H^{\bs}}\le R_1, \sup\limits_{s\in [0,t_n]}\|B(s)\|\le R_1 \}.$
Then applying the assumption $\lim\limits_{R_2 \to \infty}\lim\limits_{N\to \infty} \mathbb P((\widetilde \Omega_{R_2}^N)^c) =0$,  Proposition \ref{3d-prop}, Lemma \ref{rough-gronwall} and Corollary \ref{str-con} lead to the desired result.
\end{proof}

\begin{rk}
Based on our analysis, for a given numerical scheme for SNLSE in $d\ge 1$, the logarithmic tail estimate $\mathbb P((\Omega^N_{R_1})^c)\le C_1 \log^{-p_1}(R_1)$ implies the  double logarithmic convergence rate;  the polynomial tail estimate $\mathbb P((\Omega^N_{R_1})^c)\le C_1 R_1^{-\frac 1{p_1}}$ implies the logarithmic  convergence rate; and the exponential tail estimate implies the algebraic convergence rate. 
If there is no explicit decay rate for the tail estimate, we could use Corollary \ref{str-con} to show the strong convergence without an explicit rate. 
One can follow this approach to get the strong convergence result for a large class of SODEs and SPDEs with non-monotone coefficients.
\end{rk}

\bibliography{references}

\def\cprime{$'$} \def\cprime{$'$}
\begin{thebibliography}{10}

\bibitem{AC18}
R.~Anton and D.~Cohen.
\newblock Exponential integrators for stochastic {S}chr\"{o}dinger equations
  driven by {I}t\^{o} noise.
\newblock {\em J. Comput. Math.}, 36(2):276--309, 2018.

\bibitem{BCI95}
O.~Bang, P.~L. Christiansen, F.~If, K.~{O}~. Rasmussen, and Y.~B. Gaididei.
\newblock White noise in the two-dimensional nonlinear {S}chr\"odinger
  equation.
\newblock {\em Appl. Anal.}, 57(1-2):3--15, 1995.

\bibitem{BHJKLS19}
M.~Beccari, M.~Hutzenthaler, A.~Jentzen, R.~Kurniawan, F.~Lindner, and Salimova
  D.
\newblock Strong and weak divergence of exponential and linear-implicit {E}uler
  approximations for stochastic partial differential equations with
  superlinearly growing nonlinearities.
\newblock {\em arXiv:1903.06066}.

\bibitem{MR1691575}
J.~Bourgain.
\newblock {\em Global solutions of nonlinear {S}chr\"{o}dinger equations},
  volume~46 of {\em American Mathematical Society Colloquium Publications}.
\newblock American Mathematical Society, Providence, RI, 1999.

\bibitem{BC20}
C.~E. Br\'ehier and D.~Cohen.
\newblock Analysis of a splitting scheme for a class of nonlinear stochastic
  {S}chr\"{o}dinger equations.
\newblock {\em arXiv:2007.02354}.

\bibitem{MR3980316}
Z.~Brze\'{z}niak, F.~Hornung, and L.~Weis.
\newblock Martingale solutions for the stochastic nonlinear {S}chr\"{o}dinger
  equation in the energy space.
\newblock {\em Probab. Theory Related Fields}, 174(3-4):1273--1338, 2019.

\bibitem{MR3232027}
Z.~Brze\'{z}niak and A.~Millet.
\newblock On the stochastic {S}trichartz estimates and the stochastic nonlinear
  {S}chr\"{o}dinger equation on a compact {R}iemannian manifold.
\newblock {\em Potential Anal.}, 41(2):269--315, 2014.

\bibitem{BGT05}
N.~Burq, P.~G\'{e}rard, and N.~Tzvetkov.
\newblock Bilinear eigenfunction estimates and the nonlinear {S}chr\"{o}dinger
  equation on surfaces.
\newblock {\em Invent. Math.}, 159(1):187--223, 2005.

\bibitem{MR2002047}
T.~Cazenave.
\newblock {\em Semilinear {S}chr\"{o}dinger equations}, volume~10 of {\em
  Courant Lecture Notes in Mathematics}.
\newblock New York University, Courant Institute of Mathematical Sciences, New
  York; American Mathematical Society, Providence, RI, 2003.

\bibitem{CH16}
C.~Chen and J.~Hong.
\newblock Symplectic {R}unge--{K}utta {S}emidiscretization for {S}tochastic
  {S}chr\"odinger {E}quation.
\newblock {\em SIAM J. Numer. Anal.}, 54(4):2569--2593, 2016.

\bibitem{CHP16}
C.~Chen, J.~Hong, and A.~Prohl.
\newblock Convergence of a {$\theta$}-scheme to solve the stochastic nonlinear
  {S}chr\"odinger equation with {S}tratonovich noise.
\newblock {\em Stoch. Partial Differ. Equ. Anal. Comput.}, 4(2):274--318, 2016.

\bibitem{CH17}
J.~Cui and J.~Hong.
\newblock Analysis of a splitting scheme for damped stochastic nonlinear
  {S}chr\"{o}dinger equation with multiplicative noise.
\newblock {\em SIAM J. Numer. Anal.}, 56(4):2045--2069, 2018.

\bibitem{CHL16b}
J.~Cui, J.~Hong, and Z.~Liu.
\newblock Strong convergence rate of finite difference approximations for
  stochastic cubic {S}chr\"odinger equations.
\newblock {\em J. Differential Equations}, 263(7):3687--3713, 2017.

\bibitem{CHLZ19}
J.~Cui, J.~Hong, Z.~Liu, and W.~Zhou.
\newblock Strong convergence rate of splitting schemes for stochastic nonlinear
  {S}chr\"{o}dinger equations.
\newblock {\em J. Differential Equations}, 266(9):5625--5663, 2019.

\bibitem{CHS21}
J.~Cui, J.~Hong, and L.~Sun.
\newblock Structure-preserving splitting methods for stochastic logarithmic
  {S}chr\"odinger equation via regularized energy approximation.
\newblock {\em arXiv:2111.04402}.

\bibitem{CHS18b}
J.~Cui, J.~Hong, and L.~Sun.
\newblock On global existence and blow-up for damped stochastic nonlinear
  {S}chr\"{o}dinger equation.
\newblock {\em Discrete Contin. Dyn. Syst. Ser. B}, 24(12):6837--6854, 2019.

\bibitem{CS21}
J.~Cui and L.~Sun.
\newblock Stochastic logarithmic {S}chr\"odinger equations: energy regularized
  approach.
\newblock {\em arXiv:2102.12607}.

\bibitem{BD03}
A.~de~Bouard and A.~Debussche.
\newblock The stochastic nonlinear {S}chr\"odinger equation in {$H^1$}.
\newblock {\em Stochastic Anal. Appl.}, 21(1):97--126, 2003.

\bibitem{BD06}
A.~de~Bouard and A.~Debussche.
\newblock Weak and strong order of convergence of a semidiscrete scheme for the
  stochastic nonlinear {S}chr\"odinger equation.
\newblock {\em Appl. Math. Optim.}, 54(3):369--399, 2006.

\bibitem{DD02a}
A.~Debussche and L.~Di~Menza.
\newblock Numerical simulation of focusing stochastic nonlinear {S}chr\"odinger
  equations.
\newblock {\em Phys. D}, 162(3-4):131--154, 2002.

\bibitem{FKLT01}
G.~E. Falkovich, I.~Kolokolov, V.~Lebedev, and S.~K. Turitsyn.
\newblock Statistics of soliton-bearing systems with additive noise.
\newblock {\em Phys. Rev. E}, 63, 2001.

\bibitem{MR1949404}
D.~J. Higham, X.~Mao, and A.~M. Stuart.
\newblock Strong convergence of {E}uler-type methods for nonlinear stochastic
  differential equations.
\newblock {\em SIAM J. Numer. Anal.}, 40(3):1041--1063, 2002.

\bibitem{HJ14}
M.~Hutzenthaler and A.~Jentzen.
\newblock On a perturbation theory and on strong convergence rates for
  stochastic ordinary and partial differential equations with nonglobally
  monotone coefficients.
\newblock {\em Ann. Probab.}, 48(1):53--93, 2020.

\bibitem{MR2795791}
M.~Hutzenthaler, A.~Jentzen, and P.~E. Kloeden.
\newblock Strong and weak divergence in finite time of {E}uler's method for
  stochastic differential equations with non-globally {L}ipschitz continuous
  coefficients.
\newblock {\em Proc. R. Soc. Lond. Ser. A Math. Phys. Eng. Sci.},
  467(2130):1563--1576, 2011.

\bibitem{JWH13}
S.~Jiang, L.~Wang, and J.~Hong.
\newblock Stochastic multi-symplectic integrator for stochastic nonlinear
  {S}chr\"odinger equation.
\newblock {\em Commun. Comput. Phys.}, 14(2):393--411, 2013.

\bibitem{MR3829168}
C.~Kelly and G.~J. Lord.
\newblock Adaptive time-stepping strategies for nonlinear stochastic systems.
\newblock {\em IMA J. Numer. Anal.}, 38(3):1523--1549, 2018.

\bibitem{MR1425880}
V.~Konotop and L.~V\'{a}zquez.
\newblock {\em Nonlinear random waves}.
\newblock World Scientific Publishing Co., Inc., River Edge, NJ, 1994.

\bibitem{Liu13}
J.~Liu.
\newblock Order of convergence of splitting schemes for both deterministic and
  stochastic nonlinear {S}chr\"odinger equations.
\newblock {\em SIAM J. Numer. Anal.}, 51(4):1911--1932, 2013.

\bibitem{MR4333509}
A.~Millet, A.~D. Rodriguez, S.~Roudenko, and K.~Yang.
\newblock Behavior of solutions to the 1{D} focusing stochastic nonlinear
  {S}chr\"{o}dinger equation with spatially correlated noise.
\newblock {\em Stoch. Partial Differ. Equ. Anal. Comput.}, 9(4):1031--1080,
  2021.

\bibitem{MR3129758}
M.~V. Tretyakov and Z.~Zhang.
\newblock A fundamental mean-square convergence theorem for {SDE}s with locally
  {L}ipschitz coefficients and its applications.
\newblock {\em SIAM J. Numer. Anal.}, 51(6):3135--3162, 2013.

\bibitem{MR691044}
M.~I. Weinstein.
\newblock Nonlinear {S}chr\"{o}dinger equations and sharp interpolation
  estimates.
\newblock {\em Comm. Math. Phys.}, 87(4):567--576, 1982/83.

\bibitem{MR2428002}
S.~Zhong.
\newblock The growth in time of higher {S}obolev norms of solutions to
  {S}chr\"{o}dinger equations on compact {R}iemannian manifolds.
\newblock {\em J. Differential Equations}, 245(2):359--376, 2008.

\end{thebibliography}
\bibliographystyle{plain}

\section{Appendix}

\begin{proof}[Proof of Lemma \ref{cri-sob}]
Consider the Fourier series expansion $v=\sum\limits_{i}\<v,e_i\>e_i.$ 
The condition that $u\in \mathbb H^2$ leads to $\sum\limits_{i}\<u,e_i\>^2(1+\lambda_i^2)<\infty.$
Notice that 
$$\|u\|_{L^{\infty}}\le \|\sum\limits_{i}\<u,e_i\>e_i\|_{L^{\infty}}\le \sum_{i}|\<u,e_i\>|.$$
For $\kappa>0,$ according to the Weyl's law, the eigenvalue $\lambda_m \sim m^{\frac d2}$ with $d=2,$
 it holds that 
\begin{align*}
&\sum_{i}|\<u,e_i\>|\\
&=\sum_{\sqrt{|\lambda_i|}< \kappa}|\<u,e_i\>|+\sum_{\sqrt{|\lambda_i|}\ge \kappa}|\<u,e_i\>|\\
&\le \sum_{ \sqrt{|\lambda_i|}< \kappa}(1+\sqrt{|\lambda_i|})|\<u,e_i\>|\frac 1{1+\sqrt{|\lambda_i|}}+\sum_{\sqrt{|\lambda_i|}\ge \kappa}(1+|\lambda_i|)|\<u,e_i\>|\frac 1{1+|\lambda_i|}\\
&\le C\|u\|_{\mathbb H^1}\Big(\sum_{\sqrt{|\lambda_i|}< \kappa}\frac 1{1+|\lambda_i|}\Big)^{\frac 12}+C\|u\|_{\mathbb H^2}\Big(\sum_{\sqrt{|\lambda_i|}\ge\kappa}\frac 1{1+|\lambda_i|^2}\Big)^{\frac 12}\\
&\le C\|u\|_{\mathbb H^1}\Big(\sum_{\sqrt{|\lambda_i|}< \kappa}\frac 1{1+|\lambda_i|}\Big)^{\frac 12}+C\|u\|_{\mathbb H^2}(1+\kappa)^{-1}.
\end{align*}
Taking $\kappa= \|u\|_{\mathbb H^2},$ we complete the proof.
\end{proof}

\begin{proof}[Proof of Proposition \ref{cor-high}]
When $d=1,$ one can use the auxiliary functional $V(w)=\|(-\Delta)^{\frac \bs 2} w\|^2-\lambda \<(-\Delta)^{\bs -1} w, |w|^{2\sigma} w\>$ and follow the arguments in the proof of \cite[Theorem 2.1]{CHL16b} to show that 
\begin{align*}
\E \Big[\sup_{n\le N} \|u_n^M\|_{\mathbb H^{\bs}}^{2p}\Big]<\infty.
\end{align*}
Next, we focus on the case that $d=2$, $\sigma=1.$
Since $v_{D}^M$ preserving $\mathbb H^{\bs}$-norm for any $\bs \in \mathbb N$,
applying the It\^o formula to $\widetilde U^p(u_S^M(t))$ and using the integration by parts, we obtain
{\small
\begin{align*}
 &\widetilde U^p(u_{S}^M(t))\\
 &=\widetilde U^{p}(u_{S}^M(0))
+ \int_0^t p 2 \widetilde U^{p-1}(u_S^M (s)) \frac 1{1+\log(1+\|\Delta u_S^M (s)\|^2)}\frac 1{1+\|\Delta u_S^M (s)\|^2} \\
& \qquad \Big(\<\Delta u_S^M (s), \bi \lambda 2Re(\bar u_S^M (s) \Delta u_S^M (s))u_S^M (s) +\bi \lambda 4Re(\bar u_S^M (s) \nabla u_S^M (s))\nabla u_S^M (s)\\
&\qquad +\bi \lambda 2 |\nabla u_S^M (s)|^2 u_S^M (s)\>\Big)ds\\
&-\int_0^t p \widetilde U^{p-1}(u_S^M (s)) \frac 1{1+\log(1+\|\Delta u_S^M (s)\|^2)} \frac 1{1+\|\Delta u_S^M (s)\|^2}\\
&\qquad \sum_{i\in\mathbb N^+}\Big( \<\Delta u_S^M (s),2 |\nabla  Q^{\frac 12}e_i|^2  u_S^M (s) +2\nabla u_S^M (s) \nabla Q^{\frac 12} e_i e_i +\Delta u_S^M (s) |Q^{\frac 12}e_i|^2\\
&\qquad+2u_S^M (s)\Delta Q^{\frac 12}e_iQ^{\frac 12}e_i\> \Big)ds\\
&+\int_0^t 2p\widetilde  U^{p-1}(u_S^M (s)) \frac 1{1+\log(1+\|\Delta u_S^M (s)\|^2)}  \frac 1{1+\|\Delta u_S^M (s)\|^2}\< \Delta u_S^M (s),\bi \Delta (u_S^M (s)dW(s))\>\\
&+\int_0^t p \widetilde U^{p-1}(u_S^M (s)) \frac 1{1+\log(1+\|\Delta u_S^M (s)\|^2)} \frac 1{1+\|\Delta u_S^M (s)\|^2}\\ 
&\qquad \sum_{i\in \mathbb N^+} \Big( \|P^M \nabla u_S^M (s) \nabla Q^{\frac 12}e_i\|^2 +\|P^M \Delta u_S^M (s) Q^{\frac 12}e_i\|^2+\|P^M u_S^M (s) \Delta Q^{\frac 12}e_i\|^2\\
&\qquad +
2\< P^M u_S^M (s) \Delta Q^{\frac 12}e_i, \nabla u_S^M (s) \nabla Q^{\frac 12}e_i\>
+2\<P^M\nabla u_S^M (s)\nabla Q^{\frac 12}e_i,\Delta u_S^M (s) Q^{\frac 12}e_i\>\\
&\qquad 
+2\<P^M u_S^M (s) \Delta Q^{\frac 12}e_i,\Delta u_S^M (s) Q^{\frac 12}e_i\>\Big)ds\\
&+\int_0^t -2p \widetilde U^{p-1}(u_S^M (s))\frac 1{1+\log(1+\|\Delta u_S^M (s)\|^2)} \frac 1{(1+\|\Delta u_S^M (s)\|^2)^2}\\ 
&\qquad \Big(1+\frac 1{1+\log(1+\|\Delta u_S^M (s)\|^2)}
 -(p-1)\Big) 
\sum_{i\in \mathbb N^+}\Big(\<\Delta u_S^M (s), \bi \nabla u_S^M (s)\nabla Q^{\frac 12}e_i\>\\
&\qquad +\<\Delta u_S^M (s),\bi u_S^M (s) \Delta Q^{\frac 12}e_i\>\Big)^2ds.
 \end{align*}
 }
 Similar arguments as in the proof of Proposition \ref{d=2-h2} yield that  for a small $\epsilon\in (0,1),$ 
 \begin{align*}
 &\E \Big[\sup_{t\in [0,t_1]}\widetilde U^p(u_S^M(t)\Big]\\
 &\le \E [\widetilde U^{p}(u_S^M(0))]+\epsilon \E [\sup_{t\in [0,t_1]}\widetilde U^p(u_S^M(t)]+C(\epsilon)\int_0^{t_1} \E [\widetilde U^{p}(u_S^M(s))]ds\\
 &+C(\epsilon) \int_0^t \E \Big[ 1+\|u_S^M(s)\|^{2p} +\|\nabla u_S^M(s)\|^{2p} \Big]ds.
 \end{align*}
The Gronwall's inequality and  Proposition \ref{Sta1} yield that 
 \begin{align*}
 \E \Big[\sup_{t\in [0,T]} \widetilde U^p(u_S^M(t)\Big]\le C(T,Q,\varPsi,\lambda,\sigma,p),
 \end{align*}
 which completes the proof.
\end{proof}

\end{document}